\documentclass[11pt]{amsart}
\usepackage{amssymb, amsmath, amsthm}
\usepackage[utf8]{inputenc}

\usepackage{yhmath}

\usepackage[table,xcdraw,svgnames,dvipsnames]{xcolor}
\usepackage[colorlinks]{hyperref}
\hypersetup{
	allbordercolors=.,
	citecolor=NavyBlue,
	linkcolor=DarkRed,
	urlcolor=NavyBlue,
}

\usepackage{enumitem}

\usepackage[a4paper, centering]{geometry}

\usepackage{graphicx}
\graphicspath{{./pics/}}

\usepackage{scalerel,stackengine}

\usepackage{mathtools}

\newtheorem{theorem}{Theorem}
\newtheorem{proposition}[theorem]{Proposition}

\newtheorem{lemma}[theorem]{Lemma}
\newtheorem{corollary}[theorem]{Corollary}
\theoremstyle{remark}

\theoremstyle{definition}
\newtheorem{definition}[theorem]{Definition}
\theoremstyle{remark}
\newtheorem{remark}[theorem]{Remark}

\definecolor{verde}{RGB}{20,150,100}
\definecolor{purple}{RGB}{200,30,200}

\newcommand{\EEE}{\color{black}}

\def\R{\mathbb{R}}

\newcommand{\sm}{\setminus}

\newcommand{\Om}{\Omega}

\def \b{\beta}
\def \e{\varepsilon}

\newcommand{\abs}[1]{{\left|#1\right|}}

\begin{document}
\title[]{ 
Symmetry breaking  in the polygonal Szeg\"{o}--Weinberger inequality  as $p\to1^+$: \\ the  longest shortest-fence quadrilateral 
} 

\bigskip\bigskip

\vfill\eject 
\author[]{Beniamin Bogosel, Dorin Bucur, Ilaria Fragal\`a}

\thanks{}

\address[Beniamin Bogosel]{
	Faculty of Exact Sciences, Aurel Vlaicu University of Arad, 2 Elena Dr\u agoi Street, Arad, Romania}
\email {beniamin.bogosel@uav.ro}

\address[Dorin Bucur]{
Universit\'e  Savoie Mont Blanc, Laboratoire de Math\'ematiques CNRS UMR 5127 \\
  Campus Scientifique \\
73376 Le-Bourget-Du-Lac (France)
}
\email{dorin.bucur@univ-savoie.fr}

\address[Ilaria Fragal\`a]{
Dipartimento di Matematica \\ Politecnico  di Milano \\
Piazza Leonardo da Vinci, 32 \\
20133 Milano (Italy)
}
\email{ilaria.fragala@polimi.it}


\makeatletter
\def\@setsubjclass{%
  \itshape 2020 Mathematics Subject Classification.\enspace
  \upshape\@subjclass\@addpunct.}
\makeatother\date{\today}

\begin{abstract}    We consider P\'olya's problem of finding, among convex sets of prescribed area, the one with the longest shortest fence, in the  polygonal   setting, namely when the class of competitors is restricted to polygons with a prescribed number of sides.
While it is straightforward to show that, among triangles, the optimal shape is the equilateral one, we prove that symmetry breaking occurs in the case of quadrilaterals: the optimal quadrilateral is not the square. More precisely, we identify it as a specific isosceles trapezium,   which is uniquely determined, 
  up to homotheties and rigid motions,  by an elementary equation for its base angle.
The proof combines analytical  
  arguments 
and   rigorous interval-arithmetic computations.  
\end{abstract} 

\maketitle

\section{Introduction} 
For any convex set $\Om \subset \R^2$ with area $|\Om|$, let $m(\Om)$ denote the length of its shortest bisecting fence:
\begin{equation}\label{f:sf} 
    m(\Om )=
    \inf \Big \{\mathcal H ^ 1 (\gamma):\gamma \text{ splits } \Om \text{ into two sets of area } \frac{| \Om | }{2} \Big \}\,.
\end{equation} 
In \cite{EFKNT}, the authors proved a long-standing conjecture of P\'olya, stating that the disk uniquely maximizes the length of the shortest bisecting fence among convex sets of prescribed area or, equivalently, maximizes the scaling-invariant energy
$$
\Phi (\Om )=\frac{m(\Om )}{\sqrt{ |\Om|}}
$$
among all convex planar sets.

In this paper, we address the polygonal counterpart of this problem, which consists in maximizing the shortest fence in the class $\mathcal P_N$ of convex polygons with $N$ sides under an area constraint or, equivalently, in studying
\begin{equation}\label{f:pbpol}
\sup \Big \{  \Phi ( \Om) \ :\ \Om  \in \mathcal P _N \Big \} \,.
\end{equation}   
This purely geometric problem is very natural and is closely connected both with some long-standing open questions in spectral geometry and with some optimization problems for polygons, which have attracted increasing interest in recent years.

Indeed, if $\lambda_p ( \Om)$ and $\mu_p ( \Om)$ denote, respectively, the first Dirichlet and Neumann eigenvalues of the $p$-Laplacian, namely
$$
\lambda_p (\Omega):= \inf_{\substack{u\in W^{1, p} _0(\Om)\sm \{0\}}} 
\frac{\displaystyle{\int_\Omega \abs{\nabla u}^p} }{\displaystyle{\int_\Omega |u |^p  }} \,, \qquad 
\mu_p (\Omega):= \inf_{\substack{u\in W^{1, p}(\Om)\sm \{0\}\\ \int_\Omega  |u | ^ { p-2} u =0}} 
\frac{\displaystyle{\int_\Omega \abs{\nabla u}^p} }{\displaystyle{\int_\Omega |u |^p  }} \,,
$$ 
then, in the limit as $p \to 1^+$, $\lambda_p(\Omega)$ and $\mu_p(\Om)$ converge to the Cheeger constant $h(\Om)$ and to the shortest fence $m(\Om)$. More precisely, one has 
\smallskip

\begin{eqnarray}
& \lim \limits _{p \to 1^+} \lambda _p (\Omega)
 =
\min \limits _ {\substack{u \in BV (\R ^2)\setminus \{0\} )\\ u= 0 \hbox{ on } \R ^ 2 \setminus  \Omega  }}  
 \Big \{  \displaystyle \frac{|Du| ( \R ^2)}{\int _\Omega |u |} \Big \} = 
\min \limits _ {\substack{E\subset  \Om }}  \Big \{  \displaystyle \frac{{\rm Per} (E, \R ^2)}{|E|} \Big \} 
=
 h (\Omega)\,, & \label{f:lambdap} 
 \\  \noalign{\bigskip}
& \lim \limits _{p \to 1^+}  \mu _ p (\Om)  
=  \min  \limits _ {\substack{u \in BV  (\Om; \R^+)\\  | \{ u >0  \} |  \leq  \frac{|\Om|  }{2} }} 
  \displaystyle \frac {|D u| (\Om)  } {\int_ \Om |u|  }  =  \min \limits _ {\substack{E\subset  \Om \\  |E|   \leq \frac{|\Om| }{2}}}  
 \displaystyle \frac {{\rm Per} (E, \Om) } {|E|}  = m ( \Om) \,, & \label{f:mup}
\end{eqnarray} 
 see \cite{FrKa} for the Dirichlet case and  \cite{Gaj, GajGar, cianchi89BUMI} for the Neumann one.

In \cite{BF2}, the second and third authors proved that the regular $N$-gon minimizes the Cheeger constant among polygons with prescribed area and $N$ sides, and interpreted this result as evidence in favour of P\'olya's conjecture on the optimality of the regular $N$-gon for the first Laplacian eigenvalue 
under an area constraint (see \cite{BB} and the references therein).

Thus, at first sight, the optimality of the regular $N$-gon proved in \cite{BF2} for the Cheeger constant, together with the optimality of the disk among all convex sets obtained in \cite{EFKNT} for the shortest fence, seems to suggest that symmetry should also hold for solutions to problem \eqref{f:pbpol}.

Nevertheless, a different picture emerges if one considers the opposite limiting case, namely $p \to +\infty$. Indeed, denoting by
$$
\rho_\Om:= \max_{x \in \overline \Omega} {\rm dist}(x,\partial \Omega)
$$
the inradius of $\Om$ and by
$$
D_\Om:= \max_{{x,y} \in \overline \Omega} |x-y|
$$
its diameter, one has
$$ 
\lim _{p \to + \infty} \lambda _p   (\Omega)^ {\frac{1}{p} }  =\frac{1}{\rho _\Om}\,,  \qquad 
\lim _{p \to + \infty} \mu _ p (\Om)  ^ {\frac{1}{p} }  = \frac{2}{D_\Om}
$$  
see \cite{JuLiMa} and \cite{EKNT15}, respectively, for the Dirichlet and Neumann cases.

At this point, a striking difference appears: among polygons in $\mathcal P_N$ with fixed area, the regular polygon has maximal inradius, but in general it does not have minimal diameter. This is true for odd $N$, but false when $N \geq 6$ is even; see \cite{R22, S58} and the survey paper \cite{Moss}. For $N=6$, the optimal polygon was identified by Graham \cite{graham, B61} and is now named after him. For even $N>6$, the optimal polygon, 
known as the ``largest small polygon'', a terminology reminiscent of the ``longest shortest fence'', is still unknown, and its identification remains a challenging problem in discrete geometry; see, for instance, \cite{AHS21} and the references therein. A similar phenomenon for nonlocal perimeters was discovered in our previous paper \cite{BBF}.

From this perspective, it is reasonable to expect that solutions to problem \eqref{f:pbpol}
 behave analogously to 
  largest small polygons, and to regard   these symmetry breaking phenomena    as evidence against the optimality of the regular $N$-gon for the first Neumann eigenvalue of the $p$-Laplacian  under an area constraint, for $p$ in neighbourhoods of $1$ and $+\infty$. In contrast, 
 for $p=2$,    numerical evidence  suggests that,  when the number of sides is small,    the polygonal 
 analogue of  the Szeg\"o-Weinberger inequality holds  true.  For quadrilaterals, this has been recently proved  in \cite{EO}.

  Actually,    as regards problem \eqref{f:pbpol},   symmetry breaking unequivocally emerges from the behaviour of  the map 
\begin{equation}\label{f:mapf} N \mapsto f ( N) := \Phi ( P_N ^*)\,,
\end{equation}
where $P _N ^*$ is a regular polygon with $N$ sides. We have  (see \cite{cianchitesi, cianchi89BUMI}) 
$$
f(2m)=
\frac{2\cos\left(\frac{\pi}{2m}\right)}
{\sqrt{\frac {m}{2} \sin\left(\frac{\pi}{m}\right)}}\,,
$$
while 
$$
f(2m+1)=
\sqrt{
\frac{
\frac{2\pi}{2m+1}
\left[
\cos^2\left(\frac{3\pi}{2(2m+1)}\right)
\cot\left(\frac{\pi}{2(2m+1)}\right)
+
\frac{1}{2}\sin\left(\frac{3\pi}{2m+1}\right)
+
\frac{3}{4}\sin\left(\frac{2\pi}{2m+1}\right)
\right]
}{
\frac{2m+1}{4}\sin\left(\frac{2\pi}{2m+1}\right)
}
}.
$$
 Since a $N$-gon can be approximated by $(N+1)$-gons,  a necessary condition for the optimality of $P _N ^*$ is the increasing monotonicity of $f$, which is immediately seen to be false
by looking at small values of $N$:    
$$
\begin{array}{c|c}
N & f(N) \\
\hline
3 & 1.4472025091 \\
4 & 1.4142135624 \\
5 & 1.5429217149 \\
6 & 1.5196713713 \\
7 & 1.5688937196 \\
8 & 1.5537739740 \\
\end{array}
$$

In this paper we solve problem \eqref{f:pbpol} in the first relevant case, namely $N=4$, by identifying the optimal quadrilateral, 
which is {\it not} the square. Indeed, while for triangles it is easy to see that the solution is the equilateral one (see Corollary \ref{c:triangle} in Section \ref{sec:prel}) we prove that, for quadrilaterals, the solution is a specific isosceles trapezium. More precisely, if $\mathcal Q$ denotes the class of convex quadrilaterals in $\R^2$, we consider the optimization problem
\begin{equation}\label{f:pb}
\sup \Big \{  \Phi ( Q) \ :\ Q \in \mathcal Q \Big \} \,,
\end{equation}  
and we prove the following result:

\begin{theorem}\label{t:trapezium} 
Up to homotheties and rigid motions, problem \eqref{f:pb} admits a unique solution $Q^*$, given by the isosceles trapezium of area $1$ with acute base angle $\alpha^*$ and height $\sqrt{\alpha^*}$, where $\alpha^*$ is the unique solution in $(\frac{\pi}{3}, \frac{\pi}{2})$ of the equation
$$
\left(\frac{\pi}{2}-\alpha\right)\left(\tan^2\alpha+\alpha^2\right)=\alpha^2\tan\alpha.
$$
 Moreover,  $Q^*$ has four shortest fences (a line segment  and three circular arcs, see Figure \ref{fig0}), with length
$$
m ( Q^*)=\sqrt  {\alpha ^*   } \approx 1.049685815487323.
$$ 
\end{theorem} 

\begin{figure}
	\centering
	\includegraphics[width=0.35\textwidth]{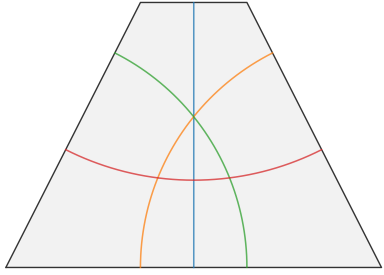}
	\caption{The quadrilateral $Q^*$ with the longest shortest fence and its four active fences. }
	\label{fig0}
\end{figure}
\bigskip

   Hereafter we give a short overview  about the proof of Theorem \ref{t:trapezium}. 

The existence of a solution to problem \eqref{f:pb} is straightforward via standard compactness and continuity arguments  (see Lemma \ref{l:existence}).   The heart of the matter is  the explicit identification of the unique solution with the 
isosceles trapezium $Q^*$. The main difficulty comes from the fact  that the shortest fence functional $m(Q)$, seen as a function of the   vertices of $Q$,   is not differentiable: indeed, it is itself defined as a minimum, and such minimum is possibly attained at several fences (that we call ``active fences''). 
Nevertheless, we can take advantage of the fact that 
the structure of active fences is quite rigid, leading in particular to   a   useful  representation formula for the length of the shortest fence (see Proposition  \ref{p:chat}); moreover, we are in a position to retrieve an optimality condition from nonsmooth analysis, namely Clarke's   necessary condition for local maximizers of a minimum (see Proposition \ref{p:clarke}).  These preliminary results, which play a crucial role in our approach, are given in Section \ref{sec:prel}.

The   structure of the  proof of Theorem \ref{t:trapezium} is then  outlined in Section \ref{sec:setup}. 
\  It consists of two parts:    we show first 
that an optimal quadrilateral is necessarily a trapezium (see Theorem \ref{t:parallel}), and  second  that,
if an optimal quadrilateral is a trapezium, it agrees necessarily with  a rescaling of $Q^*$  (see Theorem \ref{t:Qstar}). 
 For both issues, as initial step we need to  
perform a geometric analysis leading to identify, up to relabelling, 
all the possible configurations of the active fences of a candidate optimal quadrilateral.  
After presenting this analysis in Section \ref{sec:setup}, the proofs of Theorems \ref{t:parallel} and \ref{t:Qstar} are developed respectively in Sections \ref{sec:notrap} and \ref{sec:trap}.

To  prove Theorem \ref{t:parallel},  
 we assume that an optimal quadrilateral is {\it not} a trapezium,  and we show that all the possible configurations of its active fences lead to a contradiction. 
We point out that this goal is reached by means of  a  ``hybrid'' strategy.  In fact, while most cases are excluded 
via purely analytical methods, to rule out  some specific configurations we  
adopt  
an analytical argument near $Q^*$ and a numerical argument away from it.  
This must be understood in  the following sense:  we determine an explicit neighbourhood $\mathcal U$ of $Q^*$ such that,  inside $\mathcal U$,  analytical arguments show that  no quadrilateral can satisfy Clarke's  optimality condition, whereas,  outside $\mathcal U$,   certified numerical computations show that
no maximizer exists.    For clarity, the numerical part of the proof  is given in a separate section at the end of the paper. 

 Finally, to prove Theorem \ref{t:Qstar},      we proceed as follows:   assuming now that an optimal quadrilateral is a trapezium, 
we prove that    it cannot be a parallelogram,  and then we exclude all possible configurations of the active fences   in an optimal quadrilateral having exactly two parallel sides,  except for the configuration in which  two active fences have endpoints on opposite sides, and two more active fences have  endpoints on  consecutive sides. In this precise  situation, by imposing the equality of the four fences, 
we arrive at the conclusion that an optimal quadrilateral   must be  our isosceles trapezium $Q^*$.  
 
\bigskip

 We conclude this Introduction with some perspectives and open questions. 
\medskip 

-- {\it Comparison with largest small polygons, and guessing for higher $N$.} 
It is natural to compare  solutions to problem \eqref{f:pbpol}  with  largest small polygons which maximize the diameter at fixed area. 
When $N$ is odd, symmetry seems to be preserved for both problems. 
This fact, which is straightforward for triangles, 
  is a long-standing 
conjecture in case of largest small polygons, and numerical  computations suggest 
the extension of this conjecture to
 optimal polygons for problem \eqref{f:pbpol}. 
The difference occurs for $N$ even: when $N = 4$, by contrast with Theorem \ref{t:trapezium}, 
the problem of maximizing the diameter at fixed area 
is still elementary,  having infinitely many solutions (including the square) 
given by quadrilaterals with two perpendicular diagonals, both equal to the diameter (see \cite{Moss}). 
When $N = 6$,  the largest small hexagon is the Graham one  (see Figure \ref{fig:hex-oct-Graham}, left), 
while  an analytical determination of the solution to problem  \eqref{f:pbpol}
 seems quite challenging. However, 
numerical simulations show that neither the Graham hexagon nor the regular one is optimal 
 (see Figure \ref{fig:hex-oct-Graham},  centre).  
When $N \geq 8$, both problems remain open (see \cite{AHMX, BM24} and references therein). 
  For the longest shortest fence octagon, numerical simulations indicate again symmetry breaking (see Figure \ref{fig:hex-oct-Graham}, right).  Active fences (up to a tolerance 0.0001) are also indicated in the figures.

\begin{figure}
	\includegraphics[height=0.27\textwidth]{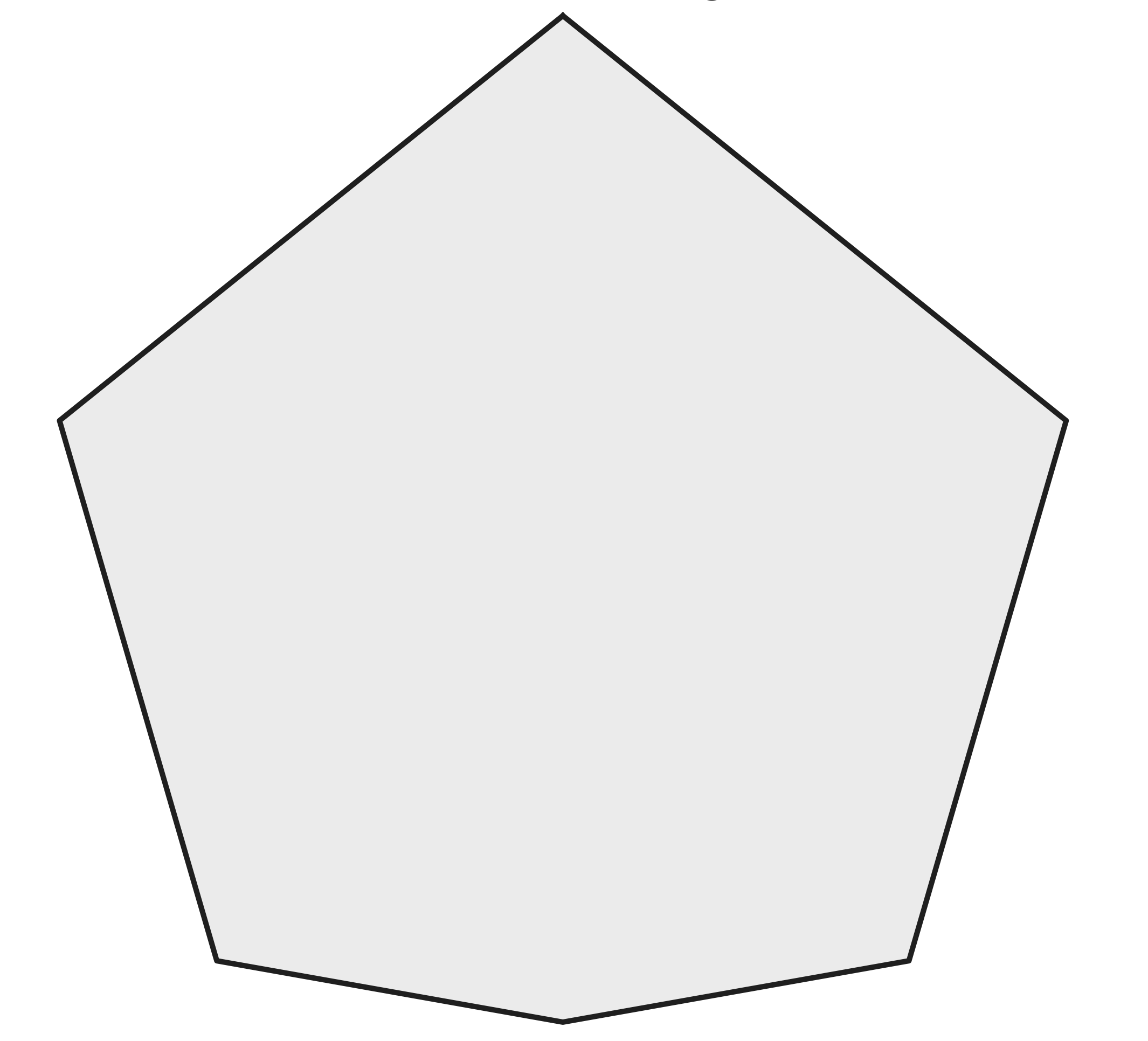}\qquad
	\includegraphics[height=0.27\textwidth]{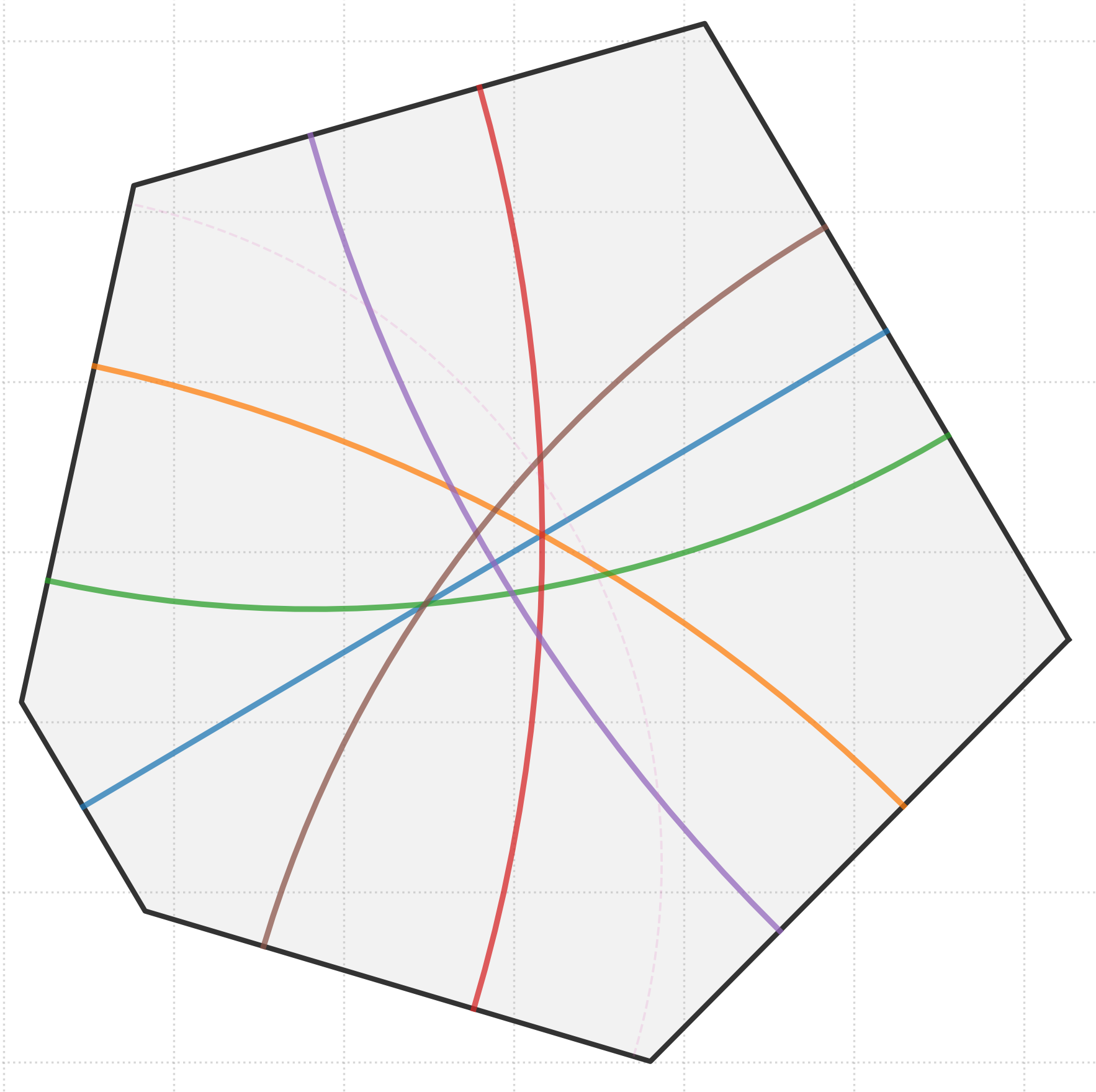}\qquad
	\includegraphics[height=0.27\textwidth]{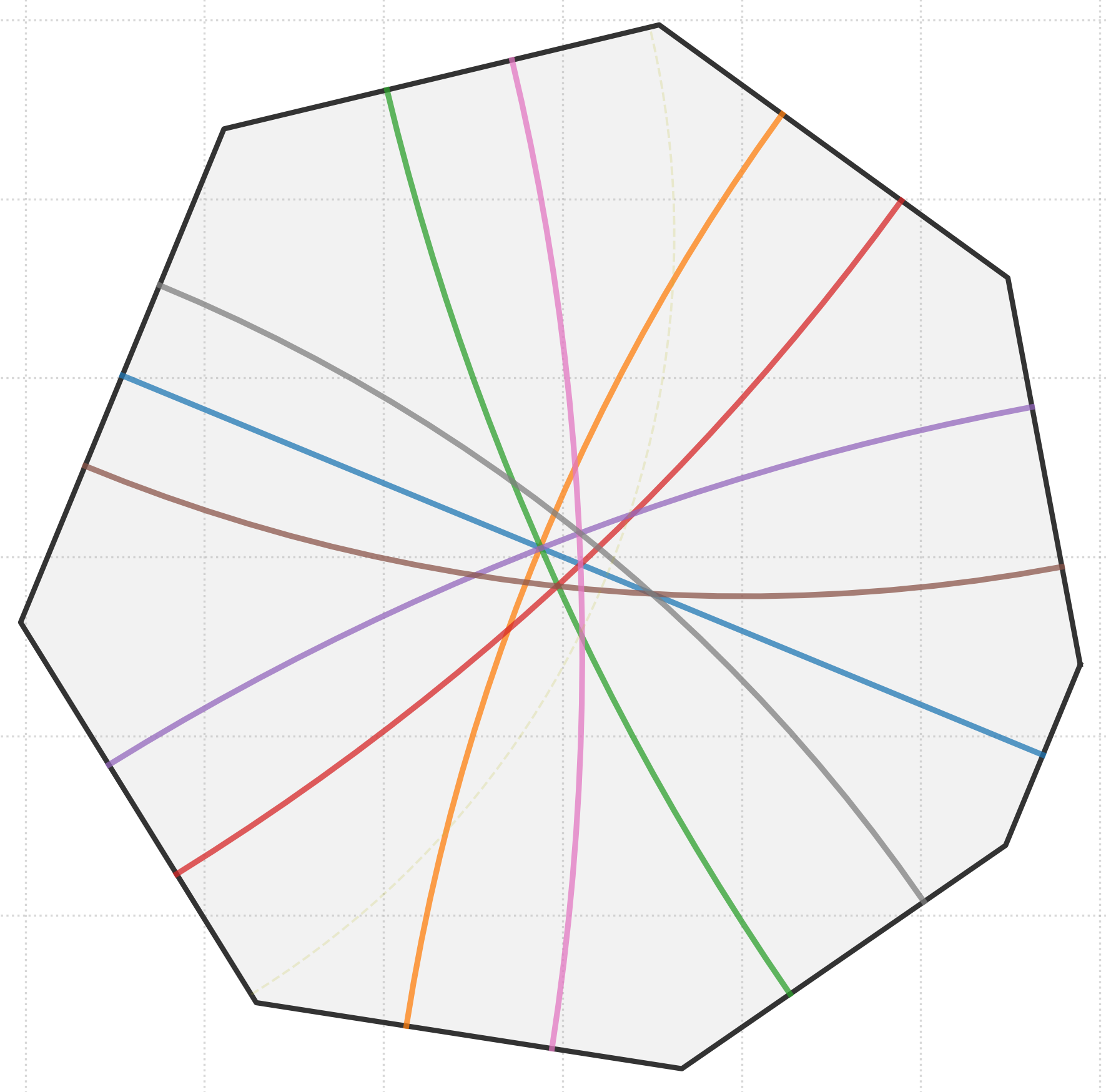}
	\caption{Graham hexagon (left)  and numerical simulations 
	for the longest shortest fence hexagon (center) and octagon  (right).  }
	\label{fig:hex-oct-Graham}
\end{figure}

\medskip

  -- {\it Polygons of given area with the longest shortest bisecting chord.}
An interesting variant of problem \eqref{f:pbpol} is obtained by considering
only bisecting chords (i.e., line segments) in the definition of the
shortest fence $m(\Om)$.
In this case, the unique maximizer of the corresponding functional $\Phi$
among all convex sets is the Auerbach triangle \cite{FP11}.
When competitors are restricted to polygons with $N$ sides, numerical
computations indicate that symmetry breaking occurs again for even $N$.
In particular, for $N=4$, the optimal quadrilateral appears to be another
isosceles trapezium, which should be possible to characterize by an approach
similar to that used in the proof of Theorem \ref{t:trapezium};
see Figure \ref{f:chords}.

 \begin{figure}
 \includegraphics[height=0.25\textwidth]{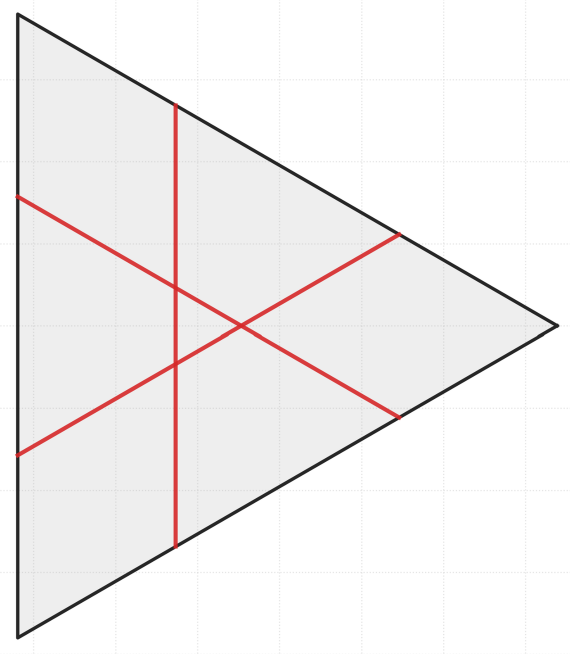} 
 \includegraphics[height=0.25\textwidth]{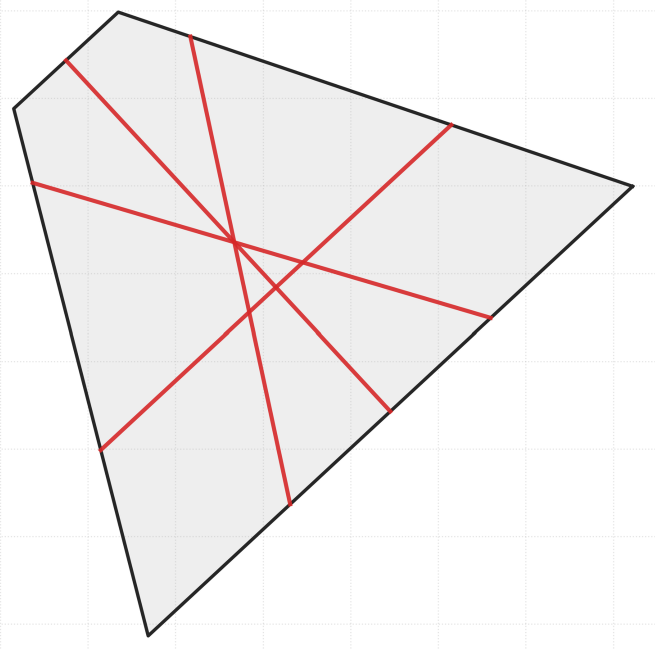} 
 \includegraphics[height=0.25\textwidth]{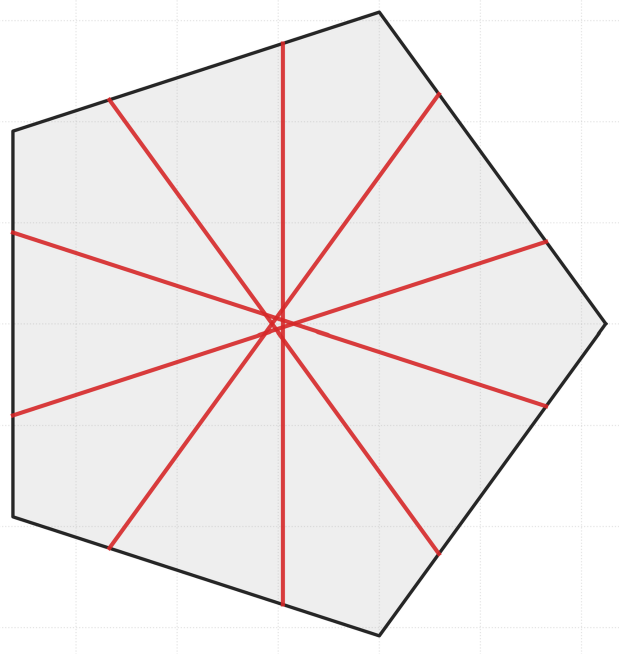} 
 \includegraphics[height=0.25\textwidth]{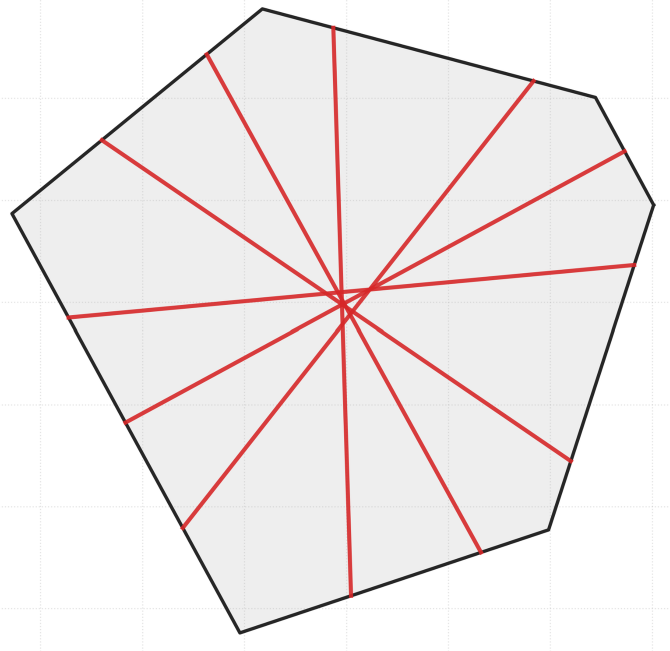} 
  \caption{ $N$-gons of given area and longest shortest bisecting chord, for   $N= 3, 4, 5, 6$.   }      
 \label{f:chords}   
 \end{figure}

\medskip

  -- {\it  Optimal partitions for  the shortest fence and for Neumann eigenvalues.}   In view of the approach developed in \cite{BFVV}, the study of problem \eqref{f:pbpol} may be useful to attack optimal partition problems in which the energy of each cell is given by its shortest fence. In turn, in light of \eqref{f:mup}, this is naturally related to optimal partition problems for Neumann eigenvalues.
Denoting by $\mathcal P_k(\Om)$ the family of all $k$-partitions of $\Om$, these problems can be written, respectively, as
$$
\max_{(\Om_i)_ i \in \mathcal P_k(\Om)} \min_ {i=1, \dots, k}  m (\Om_i)
\qquad \text{or} \qquad
\max_{(\Om_i)_ i \in \mathcal P_k(\Om)} \min_ {i=1, \dots, k} \mu_p (\Om_i).$$ 
In particular, in the linear case $p=2$, the problem for $\mu_p$ amounts to finding an optimal lower bound in Buser's inequality for the $k$-th Neumann Laplacian eigenvalue \cite[Theorem 8.2.1]{buser}.
Beyond the study of existence and regularity theory, as well as the derivation of explicit solutions for special geometries, an interesting question concerning these optimal partition problems is whether, in analogy with the Caffarelli-Lin conjecture for the first Dirichlet eigenvalue \cite{CaffLin}, the above maxima converge, after a suitable scaling, to the corresponding energy of the unit-area regular hexagon.
When the cost  of each cell is its diameter, this type of honeycomb asymptotic behavior has been conjectured in \cite{CSS}.

Gradient based simulations proposed above can be reproduced using the codes available in the repository:
\begin{center}
	\href{https://github.com/beniamin-bogosel/PolygonalFenceOptimization}{\nolinkurl{https://github.com/beniamin-bogosel/PolygonalFenceOptimization}}
\end{center}
The code contains a documentation file and instructions how to reproduce the results. 

\medskip

\section{Preliminary results}\label{sec:prel} 

 In this section we present some results 
about solutions to the  shortest fence problem
$m (\Om)$, which will be useful in order to determine the optimal quadrilateral $Q^*$ for problem \eqref{f:pb}. 
 For the sake of generality and 
since they may have deserve some independent interest (in particular 
the representation formula for the shortest fence given in Proposition \ref{p:chat}), 
 they are stated assuming that $\Om$ is a convex polygon with an arbitrary number of sides.
 At the end of the section, we also recall a useful optimality criterion from nonsmooth analysis (Proposition \ref{p:clarke}), 
which will be as well   a key tool  in the proof of Theorem \ref{t:trapezium}.

\begin{definition}  A curve $\gamma$ which is a minimizer for the shortest fence problem $m (\Om)$ in \eqref{f:sf}
will be called an   {\it active fence}  of $\Om$. 
   \end{definition}

\begin{lemma} [Existence and structure of active fences]\label{l:structure} 
For every convex polygon $\Om$,  the set of its active fences  is not empty, namely the infimum defining $m(\Om)$ is attained.
Moreover, every active fence is either a segment or an arc of circle, which meets orthogonally $\partial \Om$  away from its vertices. 
\end{lemma} 
\proof See \cite[Theorems 1 and 2]{cianchi89BUMI} (and also \cite[Propositions 1 and 2]{EFKNT}). \qed 
 
\medskip 
  \begin{lemma}[Endpoints of active fences]
For every convex polygon $\Om$,  no side  contains both the endpoints of the same active fence.
  \end{lemma}  
  
 \proof 
 Assume, by contradiction, that a side $S$ of $\Omega$ contains both endpoints of the same active fence. By Lemma \ref{l:structure}, this arc is a semicircle that cuts off from $\Om$ a half-disk of area $\frac{|\Om|}{2}$. We now translate the center of the disk along $S$ 
keeping its radius fixed. As long as the corresponding half-disk remains contained in $\Omega$, its boundary arc is still  a bisecting fence of the same length $m(\Omega)$, hence an active fence. We continue the translation up to the first contact of the circular arc with another side of $\Omega$ (in which case the contact is tangential), or with a vertex. In  both cases, this contradicts Lemma \ref{l:structure}. \qed

 \begin{lemma}[Formulas for the 
 length of active fences]\label{l:formulas}  
 Let $\Om$ be a convex polygon,   let $\gamma$ be an active fence of $\Om$,  and let
 $S'$, $S''$ be the sides hosting  the endpoints  of $\gamma$.   
  
 (i)    If $S'$ and $S''$ are consecutive, denoting by $\alpha$ the inner angle of $\Om$ at their common endpoint, it holds 
\begin{equation}\label{f:m1} 
 (\mathcal H ^ 1 (\gamma))  ^ 2 \EEE   = \alpha   |\Omega|  \,. 
\end{equation} 

\smallskip
(ii)  
If $S'$ and $S''$ are non-consecutive and non-parallel, letting $O$ be
the intersection of their supporting lines, $\theta$ be the angle
between them which contains $\Omega$, and 
\begin{equation}\label{f:omegatilde}
\widetilde\Omega
:=
\operatorname{co}(\Omega\cup\{O\})\setminus\Omega.
\end{equation} 
it holds  
\begin{equation}\label{f:m2} 
 (\mathcal H ^ 1 (\gamma))  ^ 2  
   =\theta\left(|\Omega|+2 {|\widetilde \Om|}\right)\,. 
\end{equation}

\smallskip
(iii)     If $S'$ and $S''$ are  parallel,  denoting by $d$ the distance between their support lines, it holds 
\begin{equation}\label{f:m3} 
   (\mathcal H ^ 1 (\gamma))  ^ 2 
   = d ^2\,. 
\end{equation}

\end{lemma}

 \proof  (i)    In this case by Lemma \ref{l:structure} we have that $\gamma$  is an arc of circle. 
  If $r$  is its radius,   we have
\begin{equation}
\label{f:arco} 
   \frac12\alpha r^2=\frac {|\Om| } 2,
   \qquad
  \mathcal H ^ 1 (\gamma)   =\alpha r .
\end{equation}

(ii)  Again,   we have that $\gamma$  is an arc of circle. If $r$ is its radius,   we have  
\[
   \frac12\theta r^2= \frac  {|\Om|}{2}  + |\widetilde \Om|.
\]
Since $\gamma$ has length $\theta r$, formula \eqref{f:m2} follows   by   multiplying the above identity by $2\theta$. 

  (iii) In this case by  Lemma \ref{l:structure} we have that $\gamma$ is a line segment orthogonal to $S'$ and $S''$, so the statement is immediate.     
\qed 

\bigskip
 In the next result, starting from  Lemmas \ref{l:structure} and \ref{l:formulas}, we 
provide a useful representation formula for the shortest bisecting fence:  to some extent,  it
may be regarded as an analogue of the one for the Cheeger constant of a convex polygon
obtained by Kawohl and Lachand-Robert in \cite[Theorem 3]{KaLr}.  
Their formula is global, in the sense that all sides and angles enter simultaneously, and in fact it is obtained by using the family of inner parallel sets. 
By contrast, the formula  given by Proposition \ref{p:chat}   is local, in the sense that each candidate in the minimum is associated with a pair of sides only, 
and in fact it is obtained by using  distinct families of circular arcs, each one centred at the intersection of the supporting lines of two  sides. 
  This reflects also the fact that, while 
 the Cheeger problem involves the full perimeter of the competitor, 
 in the fence problem only the relative perimeter inside $\Omega$ is counted.

\begin{proposition}[Formula for the shortest bisecting fence]\label{p:chat}
Let $\Omega$ be a convex polygon. Then
\begin{equation}\label{f:minimum}
m^2(\Omega)
= \min\left\{
\alpha_i |\Omega|,\,
\theta_j\bigl(|\Omega|+2|\widetilde\Omega_j|\bigr),\,
d_k^2
\right\}\,,
\end{equation} 
where: 
\begin{itemize}
\item{} $\alpha_i$ ranges  over the inner angles of $\Omega$;
  \item{}  $\theta_j$ ranges over the angles associated with pairs of
non-consecutive, non-parallel sides as in Lemma \ref{l:formulas} (ii),
and $\widetilde\Omega_j$ is the corresponding polygonal region as in \eqref{f:omegatilde};  
\item{} $d_k$ ranges over the distances between pairs of supporting lines containing two  parallel sides of $\Omega$.
\end{itemize} 
\end{proposition}

\begin{remark}
Before giving the proof, let us collect some comments on formula \eqref{f:minimum}.

\smallskip
\noindent
(i)  {\it Computing $|\widetilde\Omega_j|$.}      By construction, $\widetilde\Omega_j$ is a polygonal region.
Thus, if $P_0,P_1,\ldots,P_m$, with $P_{m+1}=P_0$, are its vertices,
listed in cyclic order, its area is given by the shoelace formula    
  $$ 
|\widetilde\Omega_j|=\frac12\Big|\sum_{\ell=0}^m P_\ell\times P_{\ell+1}\Big|.
$$
Inserting this identity into \eqref{f:minimum} gives a  closed formula for $m  ^2  (\Om)$ in terms of its vertices.

\smallskip
\noindent
(ii) {\it Relative isoperimetric profile.} 
Formula \eqref{f:minimum}
can be generalized to obtain a representation of  the shortest fence cutting $\Om$ in   an arbitrary    volume fraction  
$v\in(0,1)$,  i.e., 
$$ m_v (\Om) :=\inf\{\operatorname{Per}(E;\Omega):E\subset\Omega,\ |E|= v |\Omega| \}.$$
Letting $p_v:=\min\{v,1-v\}$, the same arguments used in the proof of Proposition \ref{p:chat} give 
\begin{equation}\label{f:volfrac} 
  m_v ^2  (\Om)    =
\min\left\{
2\alpha_i p_v |\Om|,\,
2\theta_j\bigl(p_v |\Om| +|\widetilde\Omega_j|\bigr),\,
d_k^2
\right\}\,, 
\end{equation}
Incidentally, this   shows  the concavity of the map $v \mapsto m_v  ^ 2  (\Om)   $ \cite[Theorem 1.1]{K03}. 

\smallskip
\noindent
(iii) {\it Beyond polygons.}    Formula \eqref{f:volfrac}
can be further extended to represent $m _ v (\Om)$ when  $\Om\subset \R ^ 2$ is  an arbitrary bounded convex  set. In fact,  the only role of the polygonal assumption in the proof 
is to make the set of admissible configurations finite. Repeating the same arguments, with pairs of sides replaced by pairs of
support  lines, gives
 $$
  m _ v ^ 2 (\Om)  
=
\inf\left\{
2\alpha {  ( L ', L '') }    p_v|\Omega|,\,
2\theta {  ( L ', L '')  } \bigl(p_v|\Omega|+|\widetilde\Omega { ( L ',L '')  }|\bigr),\,
d{  ( L', L '')   }   ^ 2
\right\},
$$ 
where  $(L', L'')$  range over pairs of support lines of 
$\Om$ such that, respectively: they meet  at a corner point of $\Om$ by forming an angle 
$\alpha ( L ', L '')  $; 
 they are non-parallel and meet at a point $O$ outside
$\Omega$, in which case $\theta(L',L'')$ denotes the angle
formed by $L'$ and $L''$ which contains $\Omega$
and
$
\widetilde\Omega(L',L'')
:=
\operatorname{co}(\Omega\cup\{O\})\setminus\Omega;
$
  they are parallel at distance 
 $d (L ', L '') $. 
\end{remark}

\begin{proof} The fact that $m ^ 2 (\Om)$ is bounded from below by the minimum in \eqref{f:minimum} is immediate. 
 Indeed,    the endpoints of an active fence may lie either on two consecutive sides, or on two non-consecutive non-parallel sides, or on two parallel sides: then by applying Lemma \ref{l:formulas} we see that the value of $m ^ 2 (\Om)$ agrees with one of the  quantities listed at the right hand side of \eqref{f:minimum}, so that it is larger than or equal to the minimum among them.

Let us now show that 
$m^2(\Omega)$ is bounded also from above by each of the quantities appearing at the r.h.s.\ of \eqref{f:minimum}. 
First, let us show the upper bound 
\begin{equation}\label{f:ubu1} 
m^2(\Omega) 
\le
\theta\bigl(|\Omega|+2|\widetilde\Omega|\bigr)
\end{equation}
 for every pair of non-consecutive, non-parallel sides.
 Let us work in    a system of polar coordinates with origin at the intersection between the supporting lines of such sides, 
 denoting by $s$ the angular variable. 
  Since $\Om$ is convex,     there exist  functions 
 $a = a ( s) $ and $b = b ( s)$  such that
$$
\begin{aligned}
\Omega
&=
\{(r,s):a(s)\le r\le b(s),\ 0<s<\theta\},\\
\widetilde\Omega
&=
\{(r,s):0\le r\le a(s),\ 0<s<\theta\}\,,
\end{aligned}
$$
 where the second identity follows from the definition of
$\widetilde\Omega$. 
Hence,
\begin{eqnarray}
& |\Omega|
 \displaystyle =
\int_0^\theta\int_{a(s)}^{b(s)} r\,dr\,ds
=
\frac12\int_0^\theta [b(s)^2-a(s)^2 ]\,ds,
& \label{f:omegaarea}
\\ \noalign{\medskip} 
& |\widetilde\Omega|
 \displaystyle =
\int_0^\theta\int_0^{a(s)} r\,dr\,ds
=
\frac12\int_0^\theta a(s)^2\,ds.
& \label{f:omegatildearea}
\end{eqnarray}
 Consequently,
$$
|\Omega|+2|\widetilde\Omega|
=
\frac12\int_0^\theta[ a(s)^2+b(s)^2 ] \,ds.
$$
We now fix  $\rho\ge0$ such that, denoting by $B(O, \rho)$ the ball of center $O$ and radius $\rho$, 
\begin{equation}\label{f:half}
| \Omega\cap B(O,\rho)| = \frac{|\Omega|}{2}\,.
\end{equation} 
Notice that such radius exists because the left hand side of the above equality is a continuous function of $\rho$, 
which starts from $0$ and eventually equals $|\Omega|$.  
Up to sets of measure zero, let us split the angular interval $(0,\theta)$ as the union of the three disjoint sets 
$$
\begin{aligned} 
& A_+=\{s:\rho\ge b(s)\},
\\ 
& M=\{s:a(s)<\rho<b(s)\},
\\ 
& A_-=\{s:\rho\le a(s)\}\,.
\end{aligned}
$$
By definition of $m (\Om)$, it is estimated from above by the length
of the circular arc
$\partial B(O,\rho)\cap\Omega$, which equals 
$\rho |M| $, where $|M|$ denotes  the Lebesgue measure of $M$ in the angular variable. 
Taking into account that $|M|\le \theta$, we have
$$
m^2 (\Omega)
\le
\rho^2 |M|^2
\le
\theta\,\rho^2 |M|
\,. $$
So, to obtain \eqref{f:ubu1}, it is enough to show that 
\begin{equation}\label{f:ubu1bis} 
 \rho^2 |M| \leq  |\Omega|+2|\widetilde\Omega| \,.
 \end{equation}
To that aim we observe that, 
by using the equality \eqref{f:omegaarea} 
and multiplying by $2$,  the half-area condition 
\eqref{f:half}  
becomes
$$
\int_{A_+} [b(s)^2-a(s)^2] \,ds
+
\int_M [\rho^2-a(s)^2] \,ds
=
\frac12\int_0^\theta [b(s)^2-a(s)^2] \,ds.
$$
Decomposing 
$(0,\theta)$ as the union $A_+\cup M\cup A_-$,
the previous identity can be rewritten as
$$
\int_M[ \rho^2-a(s)^2 ]\,ds
=
-\frac12\int_{A_+} [b(s)^2-a(s)^2 ]\,ds
+
\frac12\int_M [b(s)^2-a(s)^2 ]\,ds
+
\frac12\int_{A_-} [b(s)^2-a(s)^2] \,ds.
$$
Since 
$$
\int_M [\rho^2-a(s)^2] \,ds
=
\rho^2|M|-\int_M a(s)^2\,ds,
$$
we get
$$
\rho^2|M|
=
-\frac12\int_{A_+} [b(s) ^2-a(s) ^2 ]\,ds
+
\frac12\int_M [a(s)^2 + b ( s) ^ 2 ]\,ds
+
\frac12\int_{A_-} [b(s)^2-a(s)^2] \,ds.
$$
Dropping the non-positive term over $A_+$ and using $b^2-a^2\le b^2+a^2$ on $A_-$, we obtain
$$
\rho^2|M|
\le
\frac12\int_M [a(s)^2+b(s)^2] \,ds  \,  +
\frac12\int_{A_-}[ a(s)^2+b(s)^2 ]\,ds \leq \frac12\int_0^\theta [ a(s)^2+b(s)^2] \,ds 
=
|\Omega|+2|\widetilde\Omega|\,,
$$
yielding \eqref{f:ubu1bis} and, in turn, \eqref{f:ubu1}. 

The estimate
\begin{equation}\label{f:ubu2} 
m^2(\Omega) 
\le
\alpha |\Omega|
\end{equation}
for every 
angle $\alpha$ formed by two consecutive sides
is obtained by the same argument, now with $O$ equal to a vertex of $\Omega$. 
The corresponding sector has aperture $\alpha$, 
and the previous argument continues to work, with 
$a(s)=0$ for every $s\in(0,\alpha)$ and $|\widetilde\Omega|=0$.

Finally, it remains to show the estimate
\begin{equation}\label{f:ubu3} 
m^2(\Omega) 
\le
d ^ 2 \end{equation}
for any pair of parallel sides at distance $d$. 
 After a rigid motion, assume that the supporting lines of these sides are
$y=0$  and $y=d$,
and that $\Omega$ is contained in the strip $
\{0\le y\le d\}$.
By a continuity argument analogous to the one used above,  there exists $t$ such that 
$$
|\Omega\cap\{x<t\}|=\frac{|\Omega|}{2}.
$$ 
The relative boundary of $\Omega\cap\{x<t\}$ inside $\Omega$ is the chord
$\Omega\cap\{x=t\}$, and 
since $\Omega$ is contained in the strip $\{0\le y\le d\}$, this chord has length at most $d$. Therefore
$m(\Omega)\le d$. 
Combining the lower bound with the three upper bounds 
\eqref{f:ubu1}, \eqref{f:ubu2}, and \eqref{f:ubu3}, we obtain \eqref{f:minimum}.
\end{proof}

\bigskip
 \begin{corollary} [Optimality  of the regular triangle]\label{c:triangle}  
The unique maximizer for $\Phi$ over the class of triangles is  the regular one. 
\end{corollary} 
 
\proof By Proposition \ref{p:chat}, in any  triangle $T$ of unit area  
$m    ^ 2   ( T)$  is equal to the minimal inner angle  $\alpha _{min}$, so that   
$$m ^ 2  (T) = \alpha _{min}  \leq \frac{\pi}{3} = m ^ 2  (T ^*)\,.$$ 
\qed

\bigskip
When passing from triangles to quadrilaterals, handling the maximization of $\Phi$  becomes much more complex and will require 
the following optimality criterion:

\begin{proposition}[Clarke's condition for a local maximizer of a minimum]
\label{p:clarke}
Let $U\subset \mathbb R^k$ be open, let $x_0\in U$, and let
$f_1,\ldots,f_N:U\to\mathbb R$ be $C^1$ functions. Set
$$
\psi(x):=\min_{1\le i\le N} f_i(x),
\qquad
I(x_0):=\{i:\ f_i(x_0)=\psi(x_0)\}.
$$
If $\psi$ has a local maximum at $x_0$, then
\begin{equation}\label{f:0co} 
0\in
{\rm co} \{\nabla f_i(x_0):\ i\in I(x_0)\}.
\end{equation} 
Equivalently, there exist numbers $\lambda_i\ge 0$, for $i\in I(x_0)$, such that
$$
\sum_{i\in I(x_0)}\lambda_i=1,
\qquad
\sum_{i\in I(x_0)}\lambda_i\nabla f_i(x_0)=0.
$$
\end{proposition}

\begin{proof}
The function $-\psi$ is the finite pointwise maximum of the functions
$-f_i$. By Clarke's formula for the generalized gradient of a finite
pointwise maximum, see \cite[Proposition 2.3.12]{clarke}, and by Clarke's
necessary condition for a locally Lipschitz function to have a local
extremum, see \cite[Proposition 2.3.2]{clarke}, we get
$$
0\in \partial^C(-\psi)(x_0)
\subseteq
{\rm co}\{-\nabla f_i(x_0):\ i\in I(x_0)\}.
$$
Multiplying by $-1$ gives the desired inclusion \eqref{f:0co}. \end{proof}

\section{ Geometric setup and proof architecture of Theorem \ref{t:trapezium} }\label{sec:setup}

We first fix the labelling convention for quadrilaterals that will be used throughout the rest of the paper:

\begin{definition}[Labelling of $Q$]\label{d:labels1}
Given a convex quadrilateral $Q$,   we name cyclically its sides  in clockwise order   as 
$$S_1, \ \ S_2, \ \ S_3, \ \ S_4\,, $$
their lengths by $\ell _i:=\mathcal H ^ 1 ( S_i)$, and its vertices as
$$A = S _4 \cap S_1\, ,  \qquad B  = S _1 \cap S_2\, , \qquad 
  C  = S _2 \cap S_3\, , \qquad D = S_3 \cap S_ 4\,.$$   
 Moreover, we denote by $\gamma _{ij}$ an active fence with endpoints on  the sides
$S_i$ and $S_j$. 
 \end{definition}

The first steps towards the proof of Theorem \ref{t:trapezium} consist in establishing the existence of an optimal quadrilateral (Lemma \ref{l:existence}), 
and  
an incidence property for the endpoints of active fences (Lemma \ref{l:laws}).

\begin{lemma}\label{l:existence} 
Problem \eqref{f:pb} admits a solution.
\end{lemma}

\proof
Let $\{Q_n\}\subset \mathcal Q$ be a maximizing sequence. Up to
rescalings and rigid motions, we may assume without loss of generality
that, for every $n$, a longest side of $Q_n$ is the segment joining
$D=(0,0)$ to $C=(1,0)$ and that $Q_n$ lies in the upper half-plane.
The maximality of $|CD|$ implies that the sets $Q_n$ are contained in
a fixed compact set independent of $n$.

Thus, up to a subsequence, $\{Q_n\}$ converges in the Hausdorff distance
to a compact convex set $Q$, which may degenerate into a segment or a triangle, or be a genuine
quadrilateral. Let us exclude the first two possibilities.

To see that $Q$ has positive area, write the remaining vertices of
$Q_n$ as
$$
A_n=(a_n^x,a_n^y), \qquad B_n=(b_n^x,b_n^y).
$$
We have
$$
|Q_n|\geq
\max\left\{\frac{a_n^y}{2},\frac{b_n^y}{2}\right\}.
$$
On the other hand,  by continuity of the map
$t\mapsto |Q_n\cap\{x<t\}|$,
there exists a vertical bisecting chord of $Q_n$, whose length is at
most $\max\{a_n^y,b_n^y\}$. Hence
$$
m(Q_n)\leq \max\{a_n^y,b_n^y\}\,.
$$
Therefore
\begin{equation}\label{f:estimatearea} 
\Phi(Q_n)\leq 2\sqrt{|Q_n|}.
\end{equation} 
Since the maximizing sequence may be chosen so that
$\Phi(Q_n)\geq \Phi(Q^*)>0$,
the areas $|Q_n|$ remain bounded away from zero. By continuity of the area
under Hausdorff convergence of convex bodies, it follows that
$|Q|>0$. Thus $Q$ cannot degenerate into a segment.

We next exclude the possibility that $Q$ is a triangle. By continuity
of the shortest-fence functional and of the area under Hausdorff
convergence of convex bodies,
\begin{equation}\label{f:convphi}
\Phi(Q)=\lim_{n\to\infty}\Phi(Q_n)\,.
\end{equation}
Hence
$$
\Phi(Q)\geq \Phi(Q^*)=\sqrt{\alpha^*}>
\sqrt{\frac{\pi}{3}}.
$$
On the other hand, if $Q$ were a triangle, Corollary \ref{c:triangle}
would give
$$
\Phi(Q)\leq \sqrt{\frac{\pi}{3}},
$$
a contradiction.
Therefore $Q$ is a genuine quadrilateral. Finally,
from \eqref{f:convphi} we infer that  $Q$ realizes the maximum in \eqref{f:pb}.
\qed

\begin{lemma}[Sides hosting active fences' endpoints] 
\label{l:laws}
Let $Q$ be a  solution to problem \eqref{f:pb}. Then every side of $Q$ contains some endpoint  of an active fence. 
\EEE 

 \end{lemma}
\proof
Assume, without loss of generality, that $S_1$ does not contain any
endpoint of an active fence. We construct a perturbation of $Q$ which
increases $\Phi$, contradicting optimality.

 We shall repeatedly use Proposition \ref{p:chat}. 
Since the candidate values appearing in formula \eqref{f:minimum}
depend continuously on the vertices of $Q$, every candidate which is not
active at $Q$ gives rise, under a sufficiently small perturbation of $Q$,
to a corresponding candidate whose value remains strictly larger than
$m^2(Q)$.

We distinguish two cases.

 \begin{figure}[ht]
 \center
 \includegraphics[width=0.99\textwidth]{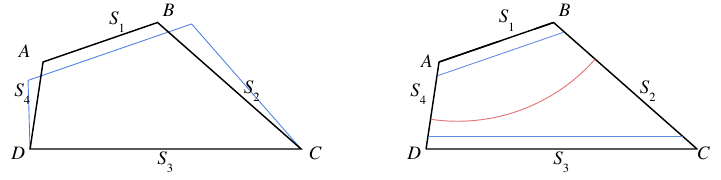}
 \caption{Perturbing $Q$ when no endpoints of active fences lie on $S_1$, and 
either there is no active fence  $\gamma _{24}$ (left) or there is such an arc (right)}      
 \label{fig1}   
 \end{figure}

{\it Case 1}: Assume  that $Q$ does not admit an active fence  of type  $\gamma _{24}$. 
Since no active fence has an endpoint on $S_1$, all the
active fences are then centred at $C$ or $D$.
We perturb $Q$ by rotating $S_2$ and $S_4$ around $C$ and $D$,
respectively, in such a way that the interior angles at $C$ and $D$
increase, and by simultaneously moving $S_1$ in the inward normal direction
so as to preserve the area (see Figure \ref{fig1}). 
Since all active fences are centred at $C$ or $D$, and the perturbation
increases the interior angles at $C$ and $D$ while preserving the area,
all candidate values which are active at $Q$ increase.
On the other hand, 
for sufficiently
small $\varepsilon>0$, the remaining ones stay strictly above $m^2(Q)$
by continuity. Hence 
$\Phi(Q_\varepsilon)>\Phi(Q)$,
a contradiction.

\smallskip

{\it Case 2}:   Assume now that $Q$ admits an active fence of type   $\gamma _{24}$. By scale invariance, we may assume without loss of generality that $|Q|=1$.
We perform simultaneous inward parallel displacements of $S_1$ and
$S_3$, chosen so that $\gamma_{24}$ still bisects the perturbed
quadrilateral $\widetilde Q_\varepsilon$ (see Figure \ref{fig1}, right),
and then rescale by setting
$Q_\varepsilon
:=|\widetilde Q_\varepsilon|^{-1/2}\widetilde Q_\varepsilon$.
Thus $|Q_\varepsilon|=1$, whereas the fence obtained by rescaling
$\gamma_{24}$ has length strictly larger than $m(Q)$, because the scaling factor is larger than $1$.

Then, we distinguish two possible subcases.

\smallskip

\smallskip
 {\it Case 2a:} $\gamma_{24}$ is the unique active fence of $Q$. Then all the other
candidate values in the representation formula \eqref{f:minimum} applied to $Q$ are strictly larger than
$m^2(Q)$. By continuity, 
for
$\varepsilon$ small,  
the corresponding candidate values for
$Q_\varepsilon$ are still strictly larger than $m^2(Q)$. On the other hand, we have observed above that the length of the bisecting
fence obtained by rescaling $\gamma_{24}$ is strictly larger than
$m(Q)$. 
Thus, all the candidate values in \eqref{f:minimum} for $Q_\varepsilon$
are strictly larger than $m^2(Q)$. Hence
$m(Q_\varepsilon)>m(Q)$,
against the optimality of $Q$.

 \smallskip
  {\it Case 2b:} $Q$ has some other active fence. Then such a fence   is necessarily of type
$\gamma_{23}$ or $\gamma_{34}$ (or both), 
since by assumption  
  no active
fence has endpoints on $S_1$. To fix the ideas, assume $Q$ has an active fence
of type $\gamma_{23}$.
 
We examine each of the candidates in formula \eqref{f:minimum} for $Q_\e$. 
First, since 
the perturbation from $Q$ to $Q_\e$ leaves unchanged
the interior angle at $C$, the candidate associated
with the consecutive sides  of $Q _\e$ obtained as perturbations of $S_2$ and $S_3$ has the
same value as the corresponding candidate in the representation formula
for $Q$, namely
$m^2(Q)$, since $\gamma_{23}$ is active for $Q$.
(The same argument applies to 
any other active fence of $Q$,  which can only
be of type $\gamma_{34}$).
  On the other hand, as observed above, the bisecting fence of
$Q_\varepsilon$ obtained by rescaling $\gamma_{24}$ has length strictly
larger than $m(Q)$. Finally, by continuity, for $\e$ small, 
all the candidates in  the representation
formula \eqref{f:minimum}  for $Q_\e$ corresponding  to fences which  not active at $Q$, remain strictly larger than $m ^ 2 (Q)$. 

It follows from Proposition \ref{p:chat}  applied to $Q_\varepsilon$
that
$m^2(Q_\varepsilon)=m^2(Q)$,
and that no active fence of $Q_\varepsilon$ is of type $\gamma_{24}$.
Therefore $Q_\varepsilon$ satisfies the assumptions of Case~1.
Hence, by applying
the perturbation constructed there to $Q_\e$, we obtain a quadrilateral
$Q_{\e,\delta}$ such that
$\Phi(Q_{\e,\delta})>\Phi(Q_\e)=\Phi(Q)$,
contradicting the optimality of $Q$.
  \qed

 \bigskip

  Now, relying on Lemmas \ref{l:existence} and \ref{l:laws}, we  proceed with the identification of the unique optimal quadrilateral with the isosceles trapezium $Q^*$. 
 As mentioned in the Introduction, this is carried over in two steps. The first and more delicate one consists in proving the statement below. 
 We specify that,  here and throughout the paper, a trapezium is meant as a quadrilateral which has {\it at least} a pair of parallel sides.  

\begin{theorem}\label{t:parallel}  Any solution to problem \eqref{f:pb} is a trapezium. 
  \end{theorem}

The proof   of Theorem \ref{t:parallel} (which will be detailed in Section \ref{sec:notrap}) is  by contradiction, and it is  of hybrid type:  assuming  that $Q$ is not a trapezium, we first determine  which are   all  the possible configurations of its active fences,  and then we shall rule them out one by one.  
In order to list them,  let us give some preparatory definitions and observations.

\smallskip

\begin{definition}[Labelling of $Q$ assuming it is not a trapezium]\label{d:labels2}  Let $Q$ be  a solution to problem \eqref{f:pb}, labelled as in Definition \ref{d:labels1}, and assume it is not a trapezium. 
Then we denote  by $P _ 1$  the intersection between the support lines of $S_1$ and $S_3$
and by $P _ 2$  the intersection between the support lines of $S_2$ and $S_4$.
Moreover, we assume that $P _ 1 $ is closer to $S_4$ than to $S_2$, and that $P_2$ is closer to $S_1$ than to $S_3$. 
Finally, we denote by $T_1$ and $T_2$ the triangles
$$T_1:= \triangle (P_1, A, D) \,, \qquad T _ 2 := \triangle (P_2, A, B )\,,$$ 
and by $\theta _1$ and $\theta _2$ the inner angles of $T_1$ and $T_2$ at $P_1$ and $P_2 $ respectively, see Figure \ref{fig2}. 
\end{definition}

\vskip - .5cm 
 \begin{figure}[ht]
 \center
 \includegraphics[width=0.54\textwidth]{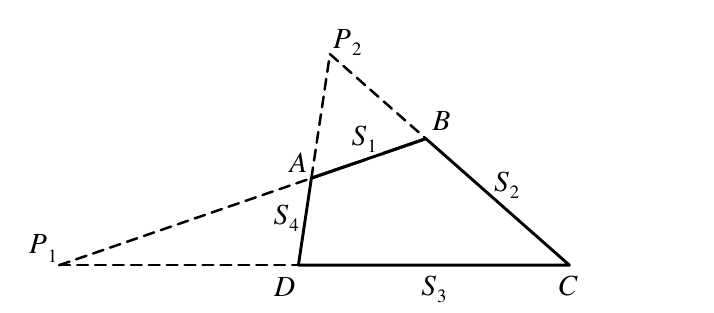}
 \caption{Labelling of $Q$ according to Definition \ref{d:labels2}.}      
 \label{fig2}   
 \end{figure}

\bigskip 

 \begin{remark}[Comparison between $\widehat A$ and $\widehat C$]\label{r:AC} 
If  $Q$  is labelled as 
in Definition \ref{d:labels2},  
  the inner angle at $A$ is strictly larger than the inner angle at $C$.   
Indeed, considering the sum of the interior angles of the triangles  $T_1$ and $T_2$, gives 
$$ ( \pi -  \widehat A) +  ( \pi -  \widehat B) +  \theta_2  = \pi  \qquad \text{ and } \qquad 
 ( \pi -  \widehat A) +  ( \pi -  \widehat D) +  \theta _1  = \pi $$ 
so that
$$ 2  \widehat A+  \widehat  B+  \widehat D=  2 \pi +  \theta _ 1 + \theta _2 \,. $$ 
Inserting into the above equality the identity  $ \widehat{B}  + \widehat{D}  = 2 \pi -  \widehat{C} - \widehat A$, coming from  
the sum of the interior angles of $ Q$, we obtain  $\widehat A =    \widehat{C} +  \theta _ 1 + \theta _2   >  \widehat{C}$. 
\end{remark}

\bigskip
 If $Q$ is labelled as 
in Definition \ref{d:labels2},  
 we call {\it active points of  $Q$}  the family $Act (Q)$ of the centres of an arc of circle which is active fence of $Q$, and we set 
 \begin{equation}\label{f:VP} 
 \mathcal V = Act (Q)  \cap \big \{ A, B, C, D \} \,, \qquad \mathcal P =Act (Q)  \cap  \big \{ P_1, P _2 \} \,.
 \end{equation} 
 Thus, a point in $\mathcal V$ is the center of an active fence with endpoints on two consecutive sides of $Q$
(so of the form $\gamma _{ij}$ with $|i - j | = 1$) 
while a point in $\mathcal P$ is the center of an active fence with endpoints on two opposite sides of $Q$
 (so of the form $\gamma _{ij}$ with $|i - j | = 2$).

\medskip
\begin{remark}[On the family $\mathcal V$ in an optimal quadrilateral which is not a trapezium] 
\label{r:V1}
Let $Q$ be a  solution to problem \eqref{f:pb} which is not a trapezium, and 
assume that the family $\mathcal V$ in \eqref{f:VP} is not empty.

Then by Proposition \ref{p:chat}  every vertex of $Q$ which belongs to $\mathcal V$ 
has inner angle of minimal amplitude $\alpha_{\min}$ among all inner angles of $Q$.

  In particular, from  the assumption that $Q$ has no parallel sides,  and from 
Remark \ref{r:AC}, we infer respectively that

  \begin{itemize}
 \item[(i)] The family $\mathcal V$ has at most $3$ elements. 
 
 \smallskip 
 \item[(ii)] If $Q$ is labelled as in Definition \ref{d:labels2}, then $\mathcal V$ does not contain  the vertex $A$. 
 \end{itemize} 
  \end{remark} 

  \smallskip 
  Now by using  Lemma \ref{l:laws},
 together with Remark \ref{r:V1}, 
we obtain the following table, up 
to cyclic relabelling of the vertices and sides of $Q$.

\medskip 
$\bullet$ {\it    Table of possible configurations of $Act (Q)$, if $Q$ is optimal and is not a trapezium:  }  
  
\smallskip
 
\begin{itemize}
\item[(${\mathcal P}0{\mathcal V}2$)]
 {\it $ {card}  (\mathcal P )  = 0$ and  $ {card}  (\mathcal V ) = 2$}. 
We have necessarily $\mathcal V = \big \{B, D \big \}$.

\medskip
\item[(${\mathcal P}0{\mathcal V}3$)]  {\it $ {card}  (\mathcal P )  = 0$ and  $ {card}  (\mathcal V ) = 3$}. 
  We have necessarily 
$
 \mathcal V  =  \{ B,  C  , D  \} 
$.  
\medskip

\item[(${\mathcal P}1{\mathcal V}2$)]  {\it $ {card}  (\mathcal P )  = 1$ and  $ {card}  (\mathcal V ) = 2$}. 
 We  have 
$\mathcal P = \{P _ 2 \}$, and 
$ \mathcal V   =  \{B, C  \}$,  or $  \{B, D\}$. 
 
\medskip
\item[(${\mathcal P}1{\mathcal V}3$)]  {\it $ {card}  (\mathcal P )  = 1$ and  $ {card}  (\mathcal V ) =  3$}. 
 We  have 
$ \mathcal V   =   \{B, C, D  \}  \,.$ 
 
\medskip

\item[(${\mathcal P}2{\mathcal V}0$)] {\it $ {card}  (\mathcal P )  = 2$ and  $ {card}  (\mathcal V ) = 0$}. 

\medskip 

\item[(${\mathcal P}2{\mathcal V}1$)]  {\it $ {card}  (\mathcal P )  = 2$ and  $ {card}  (\mathcal V ) = 1$.}
We have 
$\mathcal V = \{ D \}$ or $\{ C \}$. 

\medskip

 \item[(${\mathcal P}2{\mathcal V}2$)]  {\it $ {card}  (\mathcal P )  = 2$ and  $ {card}  (\mathcal V ) = 2$. }  
We have,  
$\mathcal V = \{ B, D \}$  or $\mathcal V = \{ 
C, D\}$. 

\medskip 
 
\item[(${\mathcal P}2{\mathcal V}3$)]  {\it $ {card}  (\mathcal P )  = 2$ and  $ {card}  (\mathcal V ) = 3$. }   
We have necessarily 
$\mathcal V = \{ B, C, D\}$.

\end{itemize} 
 
\medskip

During the proof of Theorem 
\ref{t:parallel}, all the above configurations will be excluded by exploiting the optimality conditions  supplied by Proposition \ref{p:clarke},   exception made for:

\begin{itemize}

\item[ (${\mathcal P}2{\mathcal V}1$)]  in the case when  $\mathcal V  = \{ C \}$

\medskip
\item[(${\mathcal P}2{\mathcal V}2$)]  in the case when  $\mathcal V  = \{ C, D \}$. 

\end{itemize}
 
These specific cases will be treated by invoking the following result.
It asserts that, if the two candidate values in formula \eqref{f:minimum}
associated with the pairs of opposite sides of a quadrilateral coincide,
and the quadrilateral is not close (in a quantified way) to the
isosceles trapezium $Q^*$ defined in Theorem \ref{t:trapezium},
then it cannot be optimal. 

 The
proof  of this proposition  relies on certified numerical computations, 
and it 
  is deferred to Section \ref{d11}.
   To make such proof simpler, we adopt here a different normalization: in place of fixing the area, we fix at  $1$ the length of the longest side, 
   and we work 
with no loss of generality with quadrilaterals with  longest side of vertices 
   $C=(1,0)$ and $D=(0,0)$. Then  the confidence zone consists in quadrilaterals 
 whose further vertices $A$ and $B$ do not belong to suitable rectangles centred at the vertices $A^*$ and $B^*$ of the corresponding rescaling of the
 isosceles trapezium $Q^*$ (see Figure  \ref{dfig6}).

\begin{proposition}\label{d13}
 Let $Q$ be a convex quadrilateral labelled as in Definition \ref{d:labels1}, 
  with the longest side of vertices 
   $C=(1,0)$ and $D=(0,0)$, and lying 
  in the upper half-plane. 
 Assume moreover that the two candidate values in
\eqref{f:minimum} associated with the pairs of opposite sides
$(S_1,S_3)$ and $(S_2,S_4)$ are equal.   
Consider the following rectangular neighborhoods of vertices $A,B$ of the conjectured optimal trapezium:
\[
\begin{aligned}
 \mathcal	R_A={}&[1432/4096,1500/4096]\times[2835/4096,2957/4096]\\
	={}&[0.349609375,0.3662109375]
	\times[0.692138671875,0.721923828125],\\
 \mathcal	R_B={}&[2596/4096,2662/4096]\times[2835/4096,2957/4096]\\
	={}&[0.6337890625,0.64990234375]
	\times[0.692138671875,0.721923828125].
\end{aligned}
\]
Certified computations show that if $A \notin \mathcal R_A$ or $B \notin \mathcal R_B$ then $  Q$
  is not optimal for problem \eqref{f:pb}. 
  \end{proposition}

 \begin{figure}[ht]
 \center
 \includegraphics[width=0.99 \textwidth ]{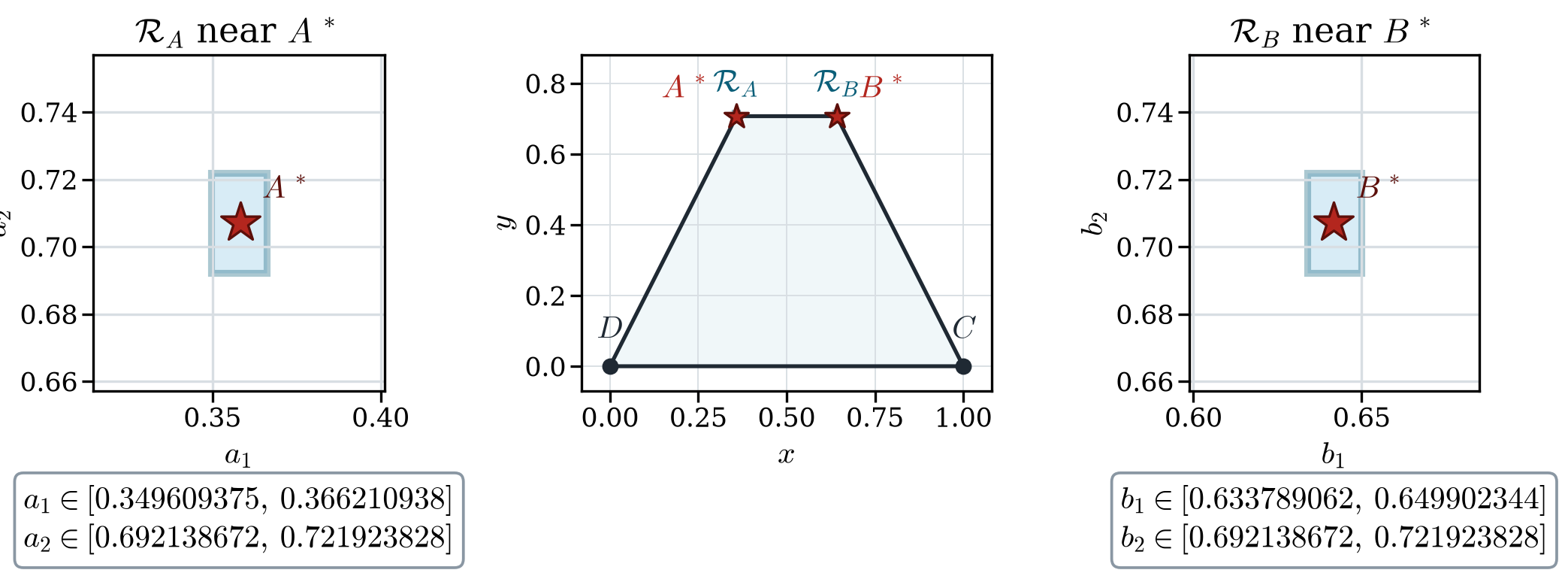}
 \caption{ The confidence rectangles $ {\mathcal R}_{A}$ and   $ {\mathcal R}_{B}$ in Proposition \ref{d13}. } 
 \label{dfig6}   
 \end{figure} 

 \medskip
 In view of Theorem \ref{t:parallel}, 
the final step to obtain  Theorem \ref{t:trapezium} consists in proving:

  \begin{theorem}\label{t:Qstar}   If a trapezium is optimal for problem \eqref{f:pb}, 
 then up to rescalings  it is the isosceles trapezium $Q^*$ defined in Theorem \ref{t:trapezium}.    \end{theorem}

 The proof  of Theorem \ref{t:Qstar} (which will be detailed in Section \ref{sec:trap}), 
 consists  again in determining first of all which are all  the possible configurations of the active fences in an optimal quadrilateral, assuming now that it is a trapezium, and then we shall rule  them out one by one, except for the one which will lead to $Q^*$. 
In order to list them,  let us make some observations.

  \begin{remark}[Parallelograms are not optimal]\label{r:rectangles} 
 If $Q$ is a parallelogram, we claim that it cannot be optimal for problem \eqref{f:pb}. Indeed, we observe first of all that 
 a Steiner symmetrization with respect to the direction orthogonal to one pair of parallel sides transforms $Q$ into a rectangle, leaves unchanged the area and   does not decrease   $m ( Q)$ (because the minimal inner angle increases,  while among the two distances between parallel sides, one is unchanged and the other one increases).
 On the other hand, 
 if $Q$ is a rectangle, then $\Phi ( Q) \leq 1$, so that $Q$ cannot be optimal for problem \eqref{f:pb}.  
 Indeed,  assuming that $|Q| = 1$, and letting $\ell$ and $\frac{1}{\ell}$ be  the lengths of its sides, by using Lemmas \ref{l:structure} and \ref{l:formulas}, we see that
 $$m^2 (Q) = \min \Big \{ \frac{\pi}{2}, \ell^2 , \frac{1}{\ell^2} \Big \} \leq 1 \,,$$
 with equality if and only if $Q$ is a square. 
 \end{remark}

 In the light of the above observation, let us fix the labelling of $Q$ under the assumption  that it is a trapezium but not a  parallelogram.

\begin{definition}[Labelling of $Q$ assuming it is a trapezium]\label{d:labels3} 
 Let $Q$ be labelled as in Definition \ref{d:labels1}, and assume it is a trapezium, with parallel sides  $S_1$ and $S_3$. 
 We still denote by $P _ 2$  the intersection between the support lines of $S_2$ and $S_4$, 
which we assume  closer to $S_1$ than to $S_3$ 
and we denote by $T_2$ the triangle
$T _ 2 := \triangle (P_2, A, B )$,  and by $\theta _2$ its inner angle at $P_2 $, see Figure \ref{fig5}. 
\end{definition}

 \begin{figure}[ht]
 \center
 \includegraphics[width=0.3\textwidth]{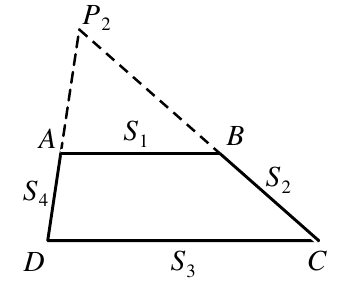}
 \caption{Labelling of $Q$ according to Definition \ref{d:labels3}.}      
 \label{fig5}   
 \end{figure}

 \begin{remark}[Comparison between $\widehat A,  \widehat C$, and between $\widehat B, \widehat D$ ]\label{r:ABCD} 
If  $Q$  is labelled as 
in Definition \ref{d:labels3}, we have
$$\widehat A > \widehat C \qquad \text{ and } \qquad \widehat B > \widehat D \,.$$  
Indeed, from the sum of the interior angles of $T_2$, we have
 $(\pi - \widehat A ) + (\pi - \widehat B) + \theta _ 2 = \pi$, so that $\widehat A + \widehat B> \pi$. Hence
 $$\widehat A > \pi - \widehat B = \widehat C \qquad \text{ and } \qquad \widehat B > \pi - \widehat A = \widehat D \,.$$ 
 \end{remark}

\bigskip 
 If $Q$ is labelled as 
in Definition \ref{d:labels3},  
 we call {\it active points of  $Q$}  the family $Act (Q)$ of the centers of an arc of circle which is active fence of $Q$, and 
 by abuse of notation,  we say that $P _ 1 ^ \infty \in Act (Q)$ if a line segment $\gamma _{13}$ (orthogonal to $S_1$ and $S_3$) 
 is an active fence of $Q$.  Then we set 
  \begin{equation}\label{f:VPbis} 
 \mathcal V = Act (Q)  \cap \big \{ A, B, C, D \} \,, \qquad \mathcal P =Act (Q)  \cap  \big \{ P_1 ^\infty, P _2 \} \,.
 \end{equation}

\medskip
 \begin{remark}[On the family $\mathcal V$ in an optimal quadrilateral which is  a trapezium] 
\label{r:V2}
Let $Q$ be a  solution to problem \eqref{f:pb} which is a trapezium, and 
assume that the family $\mathcal V$ in \eqref{f:VPbis} is not empty. 
Recalling from   Proposition \ref{p:chat}   that a   necessary condition for a vertex of $Q$ to belong to $\mathcal V$ is that 
 its inner angle has  minimal amplitude $\alpha _{min}$ among all the inner angles of $Q$, in view of Remarks \ref{r:rectangles} and \ref{r:ABCD} we have that 

  \begin{itemize}
 \item[(i)] The family $\mathcal V$ has at most $2$ elements. 
 
 \smallskip 
 \item[(ii)] If $Q$ is labelled as in Definition \ref{d:labels3}, then $\mathcal V$ contains  neither  $A$ nor $B$. 
 \end{itemize} 

    \end{remark}

\medskip 
  Now by using  Lemma \ref{l:laws},
 together with Remarks \ref{r:rectangles}   and  \ref{r:V2}, 
we obtain the following table, up 
to cyclic relabelling of the vertices and sides of $Q$.

\medskip 
$\bullet$ {\it Table of  possible configurations of  $Act(Q)$, if $Q$ is optimal and is a trapezium:}   
  
\smallskip
  \begin{itemize}
\item[(${\mathcal P}1{\mathcal V}2$)]
 {\it $ {card}  (\mathcal P )  = 1$ and  $ {card}  (\mathcal V ) = 2$}. 
We have $\mathcal P = \{P _ 1 ^ \infty\}$  and  $\mathcal V = \big \{C, D \big \}$.
 
\medskip
 
\item[(${\mathcal P}2{\mathcal V}0$)]
  {\it $ {card}  (\mathcal P )  = 2$ and  $ {card}  (\mathcal V ) = 0$}. 
We have  $\mathcal P = \{P _ 1 ^ \infty, P_2\}$.  

\medskip 

\item[(${\mathcal P}2{\mathcal V}1$)]  {\it $ {card}  (\mathcal P )  = 2$ and  $ {card}  (\mathcal V ) = 1$.}
We have  $\mathcal P = \{P _ 1 ^ \infty, P_2\}$ and  $\mathcal V = \{ D \}$. 

\medskip

\item[(${\mathcal P}2{\mathcal V}2$)]  {\it $ {card}  (\mathcal P )  = 2$ and  $ {card}  (\mathcal V ) = 2$.}  
We have  $\mathcal P = \{P _ 1 ^ \infty, P_2\}$ and   $\mathcal V = \big \{C, D \big \}$. 
\medskip

\end{itemize} 
 
 In the proof of Theorem \ref{t:Qstar},  all the above  configurations will be excluded, except for the last one, which will lead to  the conclusion that $Q = Q^*$.

  \medskip

\bigskip
\section{Proof of Theorem \ref{t:parallel}}\label{sec:notrap}

 Assume by contradiction that $Q$ is a solution to problem \eqref{f:pb} which is not a trapezium. 
  With no loss of generality, we can assume that  $|Q | = 1$.  Moreover,  denoting by $Q^*$ the trapezium introduced in Theorem 
  \ref{t:trapezium} and by $D _1$ the  disk  of area $1$, it holds  
  \begin{equation}\label{f:plage} 1.10  <  m  ^ 2 ( Q^* ) \leq m ^ 2 ( Q) < m ^ 2 ( D  _1) =  \frac{4}{\pi}  < 1.28\,,
  \end{equation} 
  where the first and the last inequalities follow from direct computation, the inequa lity $m  ^ 2 ( Q^* ) \leq m ^ 2 ( Q) $ 
  holds from the assumed optimality of $Q$, and the inequality  $m ^ 2 ( Q) < m ^ 2 (D _1)$ follows from \cite[Theorem 1]{EFKNT}. 
  
  We are going to examine separately each of the configurations 
  in the table given in Section \ref{sec:setup}, 
      and to show that all of them lead to a contradiction.  
      
 In order to exploit Proposition \ref{p:clarke}, 
it will be useful to introduce the following functions which depend on $Q$ through the cartesian coordinates of its vertices (and thus may be viewed as functions defined on $\R ^8$):
\begin{equation}\label{eq:fence-formulas}
	\begin{aligned} 
&   \phi   _{P_1}  (Q) : = \frac{1}{|Q|} \theta _ 1 (   |Q|   + 2 |T_1|)  \,,  \qquad 
  \phi    _{P_2} (Q) :=\frac{1}{|Q|}  \theta _ 2 (  |Q|     + 2 |T_2|)  
\\ 
&  \   \phi   _{A} (Q) :=  \widehat A  \EEE  \, , \quad 
   \phi   _{B} (Q):=  \widehat B  \,  , \quad  
  \phi  _{C} (Q):=    \widehat C    \, , \quad 
  \phi    _{D}(Q):=   \widehat D   \,.
\end{aligned} 
\end{equation}
By Lemma \ref{l:formulas}, we have 
$$\Phi  ^2( Q) = \min \big \{   \phi   _{P_1} (Q) ,   \phi  _{P_2}  (Q) ,    \phi  _A  (Q)  ,     \phi   _B  (Q) ,    \phi   _C  (Q) ,     \phi \EEE  _D  (Q) \big \} \,.
$$

\medskip 
\underbar {Case  (${\mathcal P}0{\mathcal V}2$)}: {\it $\mathcal P   = \emptyset$ and  $\mathcal V =  \{B, D  \}$}.

 \begin{figure}[ht]
 \center
 \includegraphics[width=0.99\textwidth]{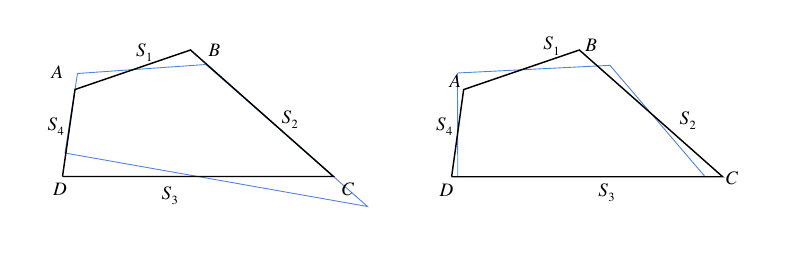} 
 \caption{On the cases $\mathcal P = \emptyset$ and $\mathcal V =  \{B, D  \} $  (left) or $\mathcal V =  \{B, C, D  \} $ (right)}      
 \label{fig3}   
 \end{figure} 
 
We consider the perturbed quadrilateral  $Q _\e$ obtained from $Q$ by performing a rotation 
of an angle $\e$ of the sides $S_1$ and $S_ 3$ around their mid-points  (see Figure \ref{fig3}, left).
Under this movement, 
 by properly choosing the directions of the rotations, 
 the inner angles 
$\widehat{ B}$ and $\widehat D$   increase at first order  in $\e$.
  On the other hand, the area of the perturbed quadrilateral satisfies $|Q_\e| = |Q| + o ( \e)  = 1 + o ( \e )  $. Hence after rescaling we have
 $\Phi \left ( \frac {Q _\e }  {  \sqrt { |Q _\e| }  } \right ) > \Phi ( Q)$. 
 
 \medskip
\underbar {Case  (${\mathcal P}0{\mathcal V}3$)}:   {\it $\mathcal P   = \emptyset$  and  $
 \mathcal V  = \{ B,  C  , D  \} 
$}. 
  
   \smallskip 
We 
 consider the perturbed quadrilateral  $Q _\e$ obtained from $Q$ by performing a rotation,  around their mid-points, 
 of an angle $\e$ of the sides $S_2$ and $S_4$,  and of an angle $\e'$ of the side $S_ 1$. 
 Under this movement, similarly as above we have that, 
by properly choosing the directions of the rotations, 
 the inner angles 
$\widehat  C$ and $\widehat  D$   increase at first order  in $\e$
  (see Figure \ref{fig3}, right).   Now,  to ensure that also the inner angle $\widehat  B$ increases at first order in $\e$,  we have to choose
  not only the direction of the rotation of $S_1$, but also the relation between $\e'$ and $\e$. 
  Let us denote by $A ^ \e, B ^ \e, C ^ \e , D ^ \e$ the vertices of $Q _\e$, by $M _i$ the midpoint of the side $S_i$ of $Q$, and by 
$P  ^\e   $ the intersection  
 between the side $[A ^ \e, B^ \e]$ of $Q _\e$ and the side $S_2$ of $Q$. 
 By comparing the inner angles of the triangles $\triangle ( P   ^\e  , M _ 1, B)$ and $\triangle (P   ^\e  , M _ 2 , B ^ \e)$, we see immediately that,
 in order to have $\widehat  B ^ \e  =\widehat  B  + \e$, it is enough to choose $\e' = 2 \e$.   
After noticing that the area of the perturbed quadrilateral still satisfies $|Q_\e| = |Q| + o ( \e)   = 1 + o ( \e)  $, we conclude that 
 $\Phi \left ( \frac {Q _\e }  {   \sqrt{ |Q _\e| } }   \right ) > \Phi ( Q)$.

\bigskip 
\underbar {Cases  (${\mathcal P}1{\mathcal V}2$) and (${\mathcal P}1{\mathcal V}3$)}:  {\it $ \mathcal P   = \{ P_2\}$ and   $\mathcal V  = \{B, C  \},  \text{ or }  \{B, D  \}, \text{ or }  \{B, C, D  \}$.
}

\smallskip 
We consider the same perturbed quadrilateral $\widetilde Q _ \e$ constructed 
in the proof of Lemma \ref{l:laws}, see Figure \ref{fig1} right.
Namely, we perform an infinitesimal parallel  movement of $S_1$ and 
$S_3$ in their normal inward directions, so that  the active fence $\gamma _{24}$ of $Q$ splits also $\widetilde  Q _\e$ into two portions of equal area. 
By rescaling $\widetilde Q_\e$, we obtain a quadrilateral  $Q_\e$ having area equal to $1$,  in which the fence $\gamma^\e _{24}$ obtained by rescaling $\gamma _{24}$ has  length strictly larger  than $\gamma _{24}$. 
 Then we argue as in Case 2b. of Lemma \ref{l:laws},  namely we observe that $Q _\e$  satisfies the geometric assumptions of 
one of the cases when ${\rm card} (\mathcal P ) = 0$ and ${\rm card} (\mathcal V ) = 2$ or $3$ already examined.

\bigskip
\underbar{Case (${\mathcal P}2{\mathcal V}0$)}:  {\it $ \mathcal P   = \{ P_1, P_2\}$  and  $\mathcal V  = \emptyset$}. 

 \smallskip
\underbar{Step 1.} We claim that in this case $Q$ is cyclic. 
 \begin{figure}[ht]
 \center
 \includegraphics[width=0.99\textwidth]{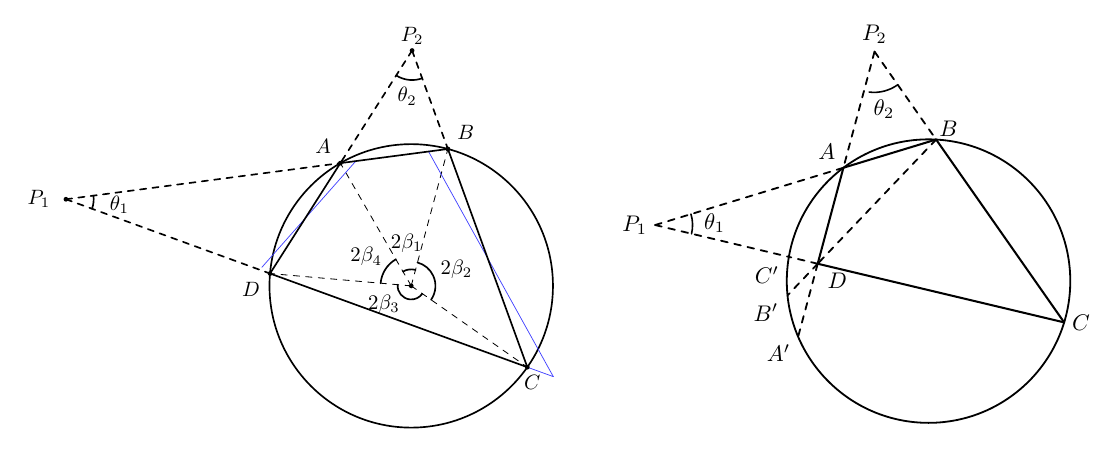}
 \caption{On the cases $\mathcal P =\{ P_1, P_2\}$ and $\mathcal V =  \emptyset$  (left) or $\mathcal V =  \{ D  \} $ (right)}      
 \label{fig4}   
 \end{figure}

Then, since by assumption $Q$ is optimal for problem \eqref{f:pb} and we have $Act ( Q) = \{ P _1, P _ 2 \}$,  
by Proposition \ref{p:clarke} there exists $\lambda _1 \geq 0, \lambda _ 2 \geq 0$, with 
$\lambda _ 1 + \lambda _ 2 = 1$,  such that  
$$\lambda _ 1 \nabla   \phi    _{P _1} (Q)  + \lambda_2 \nabla   \phi   _{P_2} (Q)  = 0\,.$$

We now consider two specific perturbations of $Q$ (and hence of  the family of its vertices in $\R^ 8$): 
we let $Q_\e$ be  either the quadrilateral 
 obtained by rotating 
 $S_2$ around its midpoint $M_2$  of a small angle $\e>0$ so that $\widehat B _ \e >  \widehat B$, 
 or the quadrilateral  obtained by rotating $S_4$ around its midpoint
  of a small angle $\e>0$, so that $\widehat A _ \e >  \widehat A$ (see Figure \ref{fig4}, left). 
 
 Since, under  these perturbations, 
neither 
$\theta _ 1 (  |Q|  + 2 |T_1|)$  nor $|Q|$ are affected at first order, 
by Proposition \ref{p:clarke} there exists a number $\lambda _ 2 \geq 0$ such that 
$$\lambda _ 2 \frac{d}{ d \e}  \big [ (\theta _ 2 + \e)  \big (   |Q _\e |    + 2 |T _{2, \e} | \big ) \big ] \Big | _{ \e = 0} = 0\,.  $$  
The possibility that $\lambda _ 2 = 0 $ is easily ruled out, since in this case Proposition \ref{p:clarke} would give 
$$\frac{d}{ d \e}  \phi_{P_1} ( Q_ \e)  \Big|_{\e=0} = 0 $$ 
  for any perturbation   $Q _ \e$   of  $Q$, which leads to a contradiction    (indeed,  such derivative is nonzero for instance  if $Q_\e$ is obtained by rotating  
  the side $S_3$ around $P_1$).

We infer that we may divide by $\lambda _2$ and obtain   
\begin{equation}\label{f:prederiv} 
 \big ( 1 + 2 |T _2 | \big )  +  \theta _ 2 \frac{d}{ d \e}   \big ( 2  |T _{2, \e} | \big )  \big | _{ \e = 0}   = 0 \,.
 \end{equation} 
Let us show that, if $Q_\e$ is obtained by rotating 
 $S_2$ around its midpoint $M_2$  so that $\widehat B _ \e >  \widehat B$, 
 it holds that  
\begin{equation}\label{f:deriv} 
\frac{d}{ d \e}   \big ( 2  |T _{2, \e} | \big )  \big | _{ \e = 0} =   -     |P _ 2 B| | P _ 2 C| \,.
\end{equation} 
Indeed, denoting by  $A _\e, B _ \e, C_\e, D _\e$  the vertices of $Q _\e$ and by
 $P _{2, \e}$, $T _ { 2, \e}$ the  perturbations of the point $P _ 2$ and of the triangle $T_2$ respectively, 
 we have
$$\begin{aligned}
 |T _{2, \e} | & = |T_2| - | P _ 2 B B _\e P _{2, \e} | = |T_2| -  \big (  
| P _ 2 M _ 2 P _ { 2, \e } | -   | B  M _ 2  B  _ { \e } |   ) 
\\ & = |T_2| -  \Big (  
\frac{1}{2} | P _ 2 M _ 2  | ^ 2  \sin \e  -  \frac{1}{2}  | B  M _ 2 | ^ 2 \sin \e \Big ) + o ( \e) \,. 
\\
& =  |T_2| -  
\frac{1}{2} \big  (   | P _ 2 M _ 2  |  +  | B  M _ 2 |   \big  )  \big (   | P _ 2 M _ 2  |  -  | B  M _ 2 |  \big   ) \,  \e + o ( \e) 
\\
& =  |T_2| -  
\frac{1}{2} \big  (   | P _ 2 M _ 2  |  +  | M _ 2 C |   \big  )  \big (   | P _ 2 M _ 2  |  -  | B  M _ 2 |  \big   ) \,  \e + o ( \e) 
\\
& =  |T_2| -  
\frac{1}{2}   | P _ 2  C |    | P _ 2  B |  \,   \e + o ( \e) \,.
  \end{aligned} 
$$ 
Inserting \eqref{f:deriv} into \eqref{f:prederiv}, we obtain 
\begin{equation}\label{f:prodotto1} 
  |P _ 2 B| | P _ 2 C | =  \frac{1}{\theta _ 2 } \big ( 1 + 2 |T _2 | \big ) \,. 
  \end{equation} 
  In a similar way, if $Q_\e$ is obtained by rotating 
 $S_4$ around its midpoint $M_4$  so that $\widehat A _ \e >  \widehat A$, 
we arrive at    \begin{equation}\label{f:prodotto2}
|P _ 2 A| | P _ 2 D | =  \frac{1}{\theta _ 2 } \big ( 1 + 2 |T _2 | \big ) \,. 
\end{equation} 
  By combining \eqref{f:prodotto1} and \eqref{f:prodotto2}, we see that 
  $$  |P _ 2 B| | P _ 2 C | = |P _ 2 A| | P _ 2 D |  \,, $$ 
 which implies that $Q$ is cyclic.  
 
 \smallskip 
\underbar{Step 2.} Denoting by $O$ the center of the  circumcircle of $Q$ and by $\beta _i $ the semi-angles with centre $O$ opposite to the chord $S_i$, i.e.: 
$$\b _ 1:= \frac{1}{2} \widehat { AOB } \, , \qquad \b _ 2:= \frac{1}{2}  \widehat   {BOC}\, , \qquad \b _ 3:=  \frac{1}{2}  \widehat   {COD} \, ,  \qquad \b _ 4:=  \frac{1}{2}  \widehat   
{DOA}\,, $$  
 it holds that
\begin{equation}\label{f:mf} m ^ 2 (Q) =   f ( \b _ 1, \b _3)= f (\b_2, \b _ 4)   \, , \quad \text { with }  \ 
  f (s, t)  : = (t - s) \frac{\sin ^ 2  ( s ) + \sin ^ 2  (t ) }{\sin ^ 2  ( t ) -\sin ^ 2  ( s )  }\,.
  \end{equation} 

Indeed, $m ^ 2 (Q)$ can be computed by inserting into  the equality \eqref{f:m2} 
the expressions of $\theta _2$ and of $|T _2|$ in terms of the angles $\beta _i$. We have 
$$\theta _ 2 = \beta _ 3- \beta _ 1   \,.$$  
On the other hand we observe that
$$\frac{ |T_2 |  }{ |T_2 | + 1 } = \Big ( \frac{ \ell _ 1}{\ell _ 3} \Big ) ^ 2  =\Big ( \frac{\sin \b _ 1 }{\sin \b _ 3} \Big ) ^ 2   \,,$$ 
so that
$$|T_ 2 | =  \frac{\sin ^ 2   \b_1   }{\sin ^ 2   \b_3  -\sin ^ 2   \b_1   } $$

Therefore,
$$\begin{aligned} m ^ 2 ( Q) 
& =  \theta _ 2 (   |Q|   + 2 |T_2 |) = (\beta _ 3- \beta _ 1 ) \Big ( 1 +  \frac{2 \sin ^ 2   \b_1   }{\sin ^ 2   \b_3  -\sin ^ 2   \b_1   }  \Big ) 
\\ 
& =  (\beta _ 3- \beta _ 1 ) \Big ( \ \frac{\sin ^ 2   \b_3  + \sin ^ 2   \b_1    }{\sin ^ 2   \b_3  -\sin ^ 2   \b_1   }  \Big )   = f ( \beta _ 1 , \beta _3) 
\,. 
\end{aligned} $$ 
The equality $m ^ 2 ( Q) = f ( \beta_2, \beta _ 4)$ is obtained by arguing in the analogous way, starting from
$m ^ 2 (Q) = \theta _ 1 ( 1 + 2 |T_1 |)$, which holds true since also $P _1$ is active.

\medskip 
\underbar{Step 3.} We have \(m^2(Q)=1\), contradicting \eqref{f:plage}. 

\smallskip 
We are going to exploit the fact that, thanks to Step 2,
  \begin{equation}\label{f:pbmax} 
m^2(Q)=f(\beta_1,\beta_3)=f(\beta_2,\beta_4)\,, \quad \text{ for some }  \b _ i > 0 \ \text{with }  \sum _{i = 1  } ^ 4 \b _ i = \pi \,,
\end{equation} 
where $f$ is the function defined in \eqref{f:mf}. 

It is convenient to rewrite $f$ in terms of the variables $u := s +t $ and $v : = s - t$, as
$$ f (s, t)   =   h ( u, v) :=   v \, \frac{1 - \cos u \cos v }{\sin u \sin v }  \,.$$ 
Letting 
$$\overline u _1 = \b _ 1 + \b _ 3\, , \qquad \overline v _ 1:=  \b _ 1 - \b _ 3 \, , \qquad \overline u _2 = \b _2 + \b _ 4\, , \qquad 
\overline v _2 = \b _ 2 - \b _ 4 \,, $$
or equivalently
\begin{equation}\label{f:betas}
\b _ 1 := \frac{\overline u _ 1 + \overline v_1}{2} \, , \qquad \b _ 2 := \frac{\overline u _ 2 + \overline v_2}{2} \, , \qquad 
\b _ 3 := \frac{\overline u _ 1 - \overline v_1}{2}  \, ,  \qquad \b _ 4 := 
\frac{\overline u _ 2 - \overline v_2}{2}\,  , 
\end{equation} 
the inequalities for $\b _i$ in \eqref{f:pbmax} can be reformulated as follows: since
$\b _1, \b _ 3 >  0$, we have $\overline u _ 1 > 0$ and $|\overline v _ 1| <  \overline u _ 1$; since $ \b _2,  \b _ 4 > 0$, 
we have $\overline u _ 2 >  0$ and $|\overline v_ 2 | <  \overline u _ 2$; moreover, 
$ \overline u _ 2 = \pi - \overline u _ 1$. 
We infer that, since $Q$ is assumed to be a nondegenerate optimal quadrilateral, the
corresponding point $(\overline u_1,\overline v_1,\overline v_2)$ belongs to the region
$$\mathcal R:=\Big \{  u _ 1 \in (0, \pi), \ |v_1 | <   u_1\, , \ |v_2|  < \pi - u _ 1 \Big \} \,,$$ 
 and it is a
stationary point for  $h$ under the constraint
$$
h(u_1,v_1)=h(\pi-u_1,v_2).
$$
 Indeed,  the
only active fences are those centred at $P_1$ and $P_2$, and $m^2(Q)=h(u_1,v_1)=h(\pi-u_1,v_2)$.
Moreover, every sufficiently small variation of the cyclic quadrilateral preserving the
relations
$\beta_i>0$, $\sum_{i=1}^4\beta_i=\pi$
corresponds, through the change of variables \eqref{f:betas}, to a variation of
$(u_1,v_1,v_2)$ in the region $\mathcal R$. Along the submanifold where the two active values
remain equal, namely
$h(u_1,v_1)=h(\pi-u_1,v_2)$
the value of $m^2(Q)$ is precisely $h(u_1,v_1)$. Since $Q$ is optimal, this function
cannot have a positive first variation along any admissible direction tangent to the
constraint. Hence $(\overline u_1,\overline v_1,\overline v_2)$ is a constrained stationary point of $h$ under the
constraint $h(u_1,v_1)=h(\pi-u_1,v_2)$.

Therefore, in order to achieve Step 3, it is enough to show that the unique  such 
constrained stationary point  is the point 
$$
\left(\frac{\pi}{2},0,0\right),
$$
where the value of $h$ is equal to $1$.

By the Lagrange multiplier theorem, letting $H ( u _1, v_1, v _2):= h ( u _ 1, v_1)$ and $\Psi ( u _1, v_1, v _2)  = 
h ( u_ 1, v_1) - h  (  \pi - u _ 1 , v_2)$, if $ (\overline u _ 1, \overline  v_1, \overline v _ 2)$
is a stationary point in  $\mathcal R$, there exists some $\lambda \in \R$ such that 
\begin{equation}\label{f:lagrange} 
\nabla H(\overline u _ 1, \overline  v_1, \overline v _ 2)  = \lambda \nabla \Psi (\overline u _ 1, \overline  v_1, \overline v _ 2) \,.
\end{equation}  
Letting
$$ a:= h _ u (\overline u _ 1, \overline  v_1 ) \, , \quad  b:= h _ v (\overline u _ 1, \overline  v_1 ), \quad
c:= h _ u (\pi - \overline u _ 1, \overline  v_2 ) \, , \quad  d:= h _ v (\pi - \overline u _ 1, \overline  v_2 )\, , $$  
  we have 
$$\nabla H(\overline u _ 1, \overline  v_1, \overline v _ 2) = (a, b , 0)
\qquad \text { and  } \qquad \nabla \Psi (\overline u _ 1, \overline  v_1, \overline v _ 2)  = (a+ c, b, -d)\,,
$$ 
so that the vector equation \eqref{f:lagrange} can be rewritten by components as
\begin{equation}\label{f:splitted} a = \lambda ( a + c) \,  , \qquad b = \lambda b \, ,  \qquad 0 = - \lambda d \,.
\end{equation}
From the explicit expression of $h _u$, which reads
$$ h _u ( u, v) =  \frac{v}{\sin v} \frac{\cos v - \cos u } {\sin ^ 2 u }\, , $$ 
we see that $a  >    0$  and $c   >   0$ for any   $ (\overline u _ 1, \overline  v_1, \overline v _ 2) \in \mathcal R$.   In particular, $\nabla \Psi \neq 0$.

We first observe that $\lambda\neq 0$. Indeed, if $\lambda=0$, then the
first equation in \eqref{f:splitted} would give $a=0$, a contradiction.
Hence, from the third equation in \eqref{f:splitted}, we get $d=0$.
Moreover, $\lambda\neq 1$, because otherwise the first equation in
\eqref{f:splitted} would give $c=0$, again a contradiction. Therefore, from
the second equation in \eqref{f:splitted}, we obtain $b=0$.

 We conclude that any critical point in $\mathcal R$ must satisfy
 \begin{equation}\label{f:2eq} h _ v (\overline u _ 1, \overline  v_1 ) =   h _ v (\pi - \overline u _ 1, \overline  v_2 ) = 0 \qquad \text{ and } \qquad 
 h ( \overline u _1, \overline v _1) = h ( \pi - \overline u _ 1, \overline  v_2 ) \,. 
 \end{equation} 
 Let us show that the unique solution is
\begin{equation}\label{f:us}
 (\overline u _ 1, \overline  v_1, \overline v _ 2)  =  \big ( \frac {\pi}{2}, 0 , 0 \big )  \,.
 \end{equation}

We  observe that
$$h _v (u,v)=
\begin{cases}
\displaystyle
\frac{1}{\sin u}
\left[
\frac{\sin v-v\cos v}{\sin^2 v}
\left(1-\cos u\cos v\right)
+
v\cos u
\right]
& \text{ if  } v\neq 0, \\ \noalign{\smallskip} 
0,
& \text{ if  } v=0\, . 
\end{cases}
$$
Hence, 
\begin{equation}\label{f:either}
h _ v (u, v)= 0 \quad \Longleftrightarrow \quad v = 0 \quad \text{ or } \quad -\cos u =  \frac{\sin v - v \cos v}{v - \sin v \cos v}\,.
\end{equation}  
We see in particular that, if  $(u, v)$ is a solution   to $h _ v (u, v)= 0$ with $v \neq 0$, it holds necessarily  $u > \frac{\pi }{2}$ (indeed,  
we have $\cos u <0$ because, for any $v\in(-\pi,\pi)\setminus\{0\}$, the numerator and the denominator
in the quotient appearing in \eqref{f:either} have the same sign).

Therefore,  in order to prove that the unique solution   $ (\overline u _ 1, \overline  v_1, \overline v _ 2)$  to  \eqref{f:2eq} is the one in \eqref{f:us}, 
it is enough to show that $\overline u _ 1 = \frac{\pi}{2}$. 
  
Assume by contradiction that $\overline u _1 \neq \frac {\pi}{2}$, and let us prove that the second equality in \eqref{f:2eq} cannot hold. 
 Just to fix the ideas assume that  $\overline u _ 1 < \frac{\pi}{2}$ 
(the other case $\pi - \overline u _ 1 < \frac {\pi}{2}$ is analogous). Then, as observed above, we have 
$\overline v _ 1 = 0$, and consequently   $$h ( \overline u _ 1 , \overline v _ 1) = h (\overline u _ 1, 0) =  \frac{1 - \cos \overline u _1  }{\sin \overline u _1 } = \tan \big ( \frac{\overline u _1}{2} \big) <1\,. $$ 
Then we distinguish two possibilities: if $\overline v _ 2 = 0$, we have 
 $$h (\pi - \overline u _ 1, \overline v _2 ) =  h (\pi - \overline u _ 1, 0) =  \frac{1 - \cos (\pi - \overline u _ 1 )  }{\sin ( \pi - \overline u _ 1 )  } = \tan \big ( \frac{\pi - \overline u _ 1 }{2} \big) >1 \,;$$
 if $\overline v _ 2 \neq 0$, from the equality $h_v (\pi - \overline u _ 1, \overline v _2 ) = 0$, 
we obtain $$
-\cos(\pi-\overline u_1) =
\frac{\sin \overline v_2-\overline v_2\cos \overline v_2}
{\overline v_2-\sin \overline v_2\cos \overline v_2}\,.$$ Inserting this expression into $ h (\pi - \overline u _ 1, \overline v _2 ) $, 
 some straightforward computations  give 
  $$h (\pi - \overline u _ 1, \overline v _2 ) =  \frac {\overline v_ 2 ^ 2} {\sqrt{ \overline v_ 2 ^ 2 - \sin ^ 2 \overline v _ 2 }} > 1 \,.$$

 In both cases we see that, since $ h ( \overline u _ 1 , \overline v _ 1)   <1$ while  $h (\pi - \overline u _ 1, \overline v _2 ) >1$, the second equality in  
 \eqref{f:2eq} cannot hold. 
 
 Hence $\overline u_1= \frac \pi 2$, and by \eqref{f:either} we get
$\overline v_1=\overline v_2=0$. 
Therefore
$$
m^2(Q)=h(\overline u_1,\overline v_1)=h\left(\frac \pi 2,0\right)=1,
$$
as claimed.

\bigskip 
\underbar{Cases (${\mathcal P}2{\mathcal V}1$)}:  We have to examine separately the subcases $\mathcal V = \{ D \}$ and $\mathcal V = \{ C \}$. 

\medskip 

$\bullet$  {\it $\mathcal P   =  \{P_1, P _2\}$ and  $ \mathcal V  = \{ D \}$. }

 \smallskip
\underbar{Step 1.} We claim that in this case $D$ lies inside the circle passing through $A, B, C$, or on the circle itself. 
 
 \smallskip
 Since $Act ( Q) = \{ P _1, P _ 2 , D\}$, by the assumed optimality of $Q$ and   
 Proposition \ref{p:clarke} there exists $\lambda _i \geq 0$, with 
$\lambda _ 1 + \lambda _ 2 + \lambda _ 3 = 1$,  such that  
$$ \lambda _ 1 \nabla  \phi    _{P _1} (Q)  + \lambda_2 \nabla   \phi  _{P_2} (Q)  + \lambda _ 3 \nabla   \phi   _D (Q) = 0\,.  $$

Also in this case, we are going to consider    two specific perturbations of $Q$ (and hence of  the family of its vertices in $\R^ 8$).

As a first perturbation, we  let $Q _ \e$ be  the quadrilaterals obtained by rotating  $S_2$ around its midpoint. 
Under such perturbation, $\widehat D$ is unchanged, and 
neither 
$\theta _ 1 (  |Q|   + 2 |T_1|)$  nor $|Q|$ are affected at first order. 
Then  we have
$$\lambda _ 2 \frac{d}{ d \e}  \big [ (\theta _ 2 + \e)  \big (  |Q_\e |    + 2 |T _{2, \e} | \big ) \big ] \Big | _{ \e = 0} = 0\,.  $$  
 
 We claim that $\lambda _ 2 \neq 0$. Indeed, assume by contradiction that $\lambda _ 2 = 0$. 
 Then, consider the perturbations of $Q$ obtained by rotating $S_4$ around its midpoint, so that $\widehat D$ is changed into 
 $\widehat D + \e > \widehat D$: since under this perturbation
 $ \phi   _{P_1}$ is unchanged at first order, while the first variation of $  \phi   _{D}$ is strictly positive, 
it follows that $\lambda _ 3 = 0$. But then the simultaneous validity of the  equalities $\lambda _ 2 = \lambda _ 3 = 0$ 
 implies that  $\lambda_1=1$ and hence  that $\nabla\phi_{P_1}(Q)=0$. This gives a contradiction by the same
perturbation argument used in the case
$({\mathcal P}2{\mathcal V}0)$ (where we considered the quadrilateral obtained by rotating $S_3$
around $P_1$).  

Thus we can divide by $\lambda_2$ and,   
    by arguing as in the case analysed above when ${\rm card} (\mathcal P ) = 2$ and  ${\rm card} (\mathcal V ) = 0$, 
we arrive at  the equality  
\begin{equation}\label{f:prodotto1bis} 
  |P _ 2 B| | P _ 2 C | =  \frac{1}{\theta _ 2 } \big ( 1 + 2 |T _2 | \big ) \,. 
  \end{equation} 
  
  As a second perturbation, we  let $Q _ \e$ be  the quadrilaterals obtained by rotating  $S_4$ around its midpoint, 
  so that $\widehat D$ becomes $\widehat D + \e$.  Since as observed above $  \phi    _{P_1}$ remains unchanged at first order, 
  we have 
  
  $$\lambda _2 \frac{d}{ d \e}  \big [ (\theta _ 2 - \e)  \big (   |Q_\e |    + 2 |T _{2, \e} | \big ) \big ] \Big | _{ \e = 0} + \lambda _ 3 =   0\,,  $$ 
 
Taking into account that we have already shown that $\lambda _2 \neq 0$, we infer that 

  $$\frac{d}{ d \e}  \big [ (\theta _ 2 - \e)  \big (  |Q_\e |    + 2 |T _{2, \e} | \big ) \big ] \Big | _{ \e = 0}  = - \frac{  \lambda _ 3 }{\lambda _ 2} \leq 0  
  \,.  $$ 
   
In view of the asymptotic expansion $$T _ { 2, \e} = |T_2| -  
\frac{1}{2}   | P _ 2  A |    | P _ 2  D |  \,   \e + o ( \e)$$  
(which is obtained as in the case ${\rm card} (\mathcal P ) = 2$ and  ${\rm card} (\mathcal V ) = 0$,   just changing the sign of $\e$), 
the above inequality leads to 
\begin{equation}\label{f:prodotto2bis}
|P _ 2 A| | P _ 2 D | \leq  \frac{1}{\theta _ 2 } \big ( 1 + 2 |T _2 | \big ) \,. 
\end{equation} 
 
By comparing \eqref{f:prodotto1bis}  and \eqref{f:prodotto2bis}, we obtain that
$$ |P _ 2 A| | P _ 2 D | \leq   |P _ 2 B| | P _ 2 C |  \,.$$ 
This means that $D$ lies inside the circle passing through $A, B, C$ (if the inequality is strict),  or on the circle itself
(if equality holds). 

 \smallskip
\underbar{Step 2.} Reaching a contradiction. 
 
 \smallskip 
Assume first  that $D$ lies inside the circle through $A, B, C$: let us extend  the segments 
$[A, D], [B, D], [C, D]$ until they meet the circle at points (distinct from $A, B, C$) that we denote respectively by $A'$, $B'$, $C'$, 
see Figure \ref{fig4}, right.  
 
 We observe that
$$\begin{aligned}
& \widehat D = \angle {ADC} = \frac{1}{2} \big (\wideparen{AB} + \wideparen{BC}  + \wideparen{C'B'}    + \wideparen{B' A'} \big ) 
\\ 
& \widehat C = \angle {BCD} = \frac{1}{2}  \big  (\wideparen{AB} + \wideparen {AC'} \big ) \,. 
\end{aligned}
$$ 
Here the first equality follows from the theorem on two chords intersecting inside the circle, 
since $D$ is the intersection of the chords $[A,A']$ and $[C,C']$. 
The second one follows from the equality $\angle BCD=\angle BCC'$, 
and from the inscribed angle theorem applied to $\angle BCC'$, whose intercepted arc is $\wideparen{BC'}$.

Taking into account that $P_2$ is the intersection of the two secants
$P_2AB$ and $P_2C'C$, by the exterior angle theorem we have 
$$\theta _ 2 = \frac{\wideparen{BC}- \wideparen{AC'}}{2}\,. 
$$  
Since $\theta _ 2 >0$, by comparing the above expressions of $\widehat D$ and $\widehat C$, 
we deduce that $\widehat D  > \widehat C$.  
This 
 contradicts 
Remark \ref{r:V1}, due to  the assumption that $D$ is active while $C$ is not.

It remains to consider the case where $D$ lies on the circle through $A,B,C$, so that $Q$ is cyclic. 
In this case, by the inscribed angle theorem, we have 
$$
\begin{aligned} 
& \widehat D=\angle ADC
=
\frac12\bigl(\wideparen{AB}+\wideparen{BC}\bigr),
 \\ 
 & \widehat C=\angle BCD
=
\frac12\bigl(\wideparen{AB}+\wideparen{AD}\bigr)\, , 
\end{aligned} 
$$
where the arcs are those not containing the vertex of the corresponding inscribed angle. 
We infer that 
$$
\widehat D-\widehat C
=
\frac12\bigl(\wideparen{BC}-\wideparen{AD}\bigr)
=
\theta_2
>
0.
$$
where in the second equality we have applied as above the exterior angle theorem. 

Thus, we have $\widehat D>\widehat C$. 
This 
 contradicts  again 
Remark \ref{r:V1}, due to  the assumption that $D$ is active while $C$ is not.

\bigskip 
$\bullet$  {\it $\mathcal P   =  \{P_1, P _2\}$ and  $ \mathcal V  = \{ C \}$. } 
\smallskip

\underbar{Step 1}. It holds that
\begin{equation}\label{f:DC}  |DC| ^ 2 = 
\frac{(\pi - \widehat D) \,  \sin ( \widehat C +  \widehat D)     }{  (\pi - \widehat C  - \widehat D) \, \sin \widehat D \, \sin \widehat C }\,.
\end{equation}

Indeed, the area of the triangle $\triangle(P_2, C, D)$ can be expressed either 
in terms of $|DC|$ and its inner angles by elementary geometry, or in terms of the angle $\theta_2$ and of
the radius $r_2$ of $\gamma _{24}$, by taking into account that  $\gamma _{24}$ bisects  $Q$. This gives 
$$ 
| \triangle(P_2, C, D) |  = |DC| ^ 2 \frac{\sin \widehat C \sin \widehat D}{2  \sin (\widehat C + \widehat D)} = 
\frac {1}{2}  \theta _ 2   r_ 2 ^ 2  + \frac{1}{2} \,. $$ 
We infer that 
\begin{equation}\label{f:lato} 
|DC| ^ 2  = \big (  \theta _ 2   r_ 2 ^ 2  + 1\big )  \frac{ \sin (\widehat C + \widehat  D)}{\sin \widehat C \sin \widehat D} 
\,.
\end{equation}
On the other hand, we have
$$\theta _ 2 r _ 2 = \mathcal H ^ 1 (\gamma _{24}  ) = \mathcal H ^ 1 (\gamma _ { 23}) = \widehat C ^ {\frac 1 2}\,,$$
where the first equality is just the definition of $r_2$, 
the second equality holds since $P_2$ and $C$ are simultaneously active, and the third one is due to Lemma \ref{l:formulas} (i).  
From the sum of the interior angles of  $\triangle(P_2, C, D) $, we see that $\theta _ 2   = \pi - \widehat C - \widehat D$. Inserting the equality
$$ \theta _ 2 r  _2 ^ 2 = \frac{ ( \theta _ 2 r _ 2   ) ^ 2 }{\theta _2} = \frac{  \widehat C  }{\pi - \widehat C - \widehat D }$$ 
into \eqref{f:lato},  we obtain   \eqref{f:DC}.

\medskip 
\underbar{Step 2}. We have
\begin{equation}\label{f:mdel}
m ^ 2 ( Q) = f (\widehat B, \widehat C, \widehat D): =  \theta _ 1 \Big ( 1 + x ^ 2 \frac{\sin \theta _ 1\,  \sin \widehat D}{\sin ( \widehat D - \theta _ 1)} \Big ) \,,
\end{equation} 
where
$\theta _ 1 = \pi - (\widehat B + \widehat C) $, 
  and $x$ is the  smallest positive  solution to the second order equation  
\begin{equation}\label{f:degree2} 
(m-n) x ^ 2  +( 2  m \ell _ 3 )  x + (m \ell _ 3 ^ 2  - 2)  = 0 \,,
\end{equation} 
being $\ell _ 3 =  |DC|$ given by \eqref{f:DC}, and 
$$m=m (\theta_1, \widehat C) = \frac{\sin \theta _1 \, \sin \widehat C}{\sin (\theta _ 1 + \widehat C)} 
 \, , \qquad n:=  n ( \theta _ 1, \widehat D)= \frac{\sin \theta _ 1 \, \sin \widehat D}{\sin (\widehat D - \theta _ 1)}  \,.  $$ 
Indeed, if we let  $x:=  |PD|$, the function $f (\widehat B, \widehat C, \widehat D)$ defined  in \eqref{f:mdel} corresponds exactly to 
$\theta _ 1 ( 1 + 2 |T_1|)$, which by Lemma \ref{l:formulas} (ii) gives $m ^ 2(Q)$, provided $|Q| = 1$.  Note that the expressions of $x$ and of the function $f$ are well defined and have proper meaning for angles $ (\widehat B, \widehat C, \widehat D)$ in a neighbourhood of the optimal quadrilateral. 
Thus,  to achieve Step 2 
we have to show just that the choice of $x$ as the 
smallest  positive solution to \eqref{f:degree2} is the good one which ensures that $|Q| = 1$.    This is readily checked, since imposing that $|Q|= 1$ we 
get
$$1 = |Q| = |\triangle(P_1, B, C)- \triangle (P_1, A, D)|  = \frac{m}{2} (x+\ell _3) ^ 2 - \frac{n}{2} x ^ 2 \,,$$ 
 which corresponds to \eqref{f:degree2}. 

\smallskip
\underbar{Step 3.} Let us show that 
\begin{eqnarray} 
 &   |P _ 2 A| | P _ 2 D | = \displaystyle   \frac{1}{\theta _ 2 } \big ( 1 + 2 |T _2 | \big )  & \label{f:2e} 
 \\ \noalign{\smallskip} 
& 
\widehat B +  \widehat D \leq \pi\,. & \label{f:3e}  
\end{eqnarray} 
  Since $Act ( Q) = \{ P _1, P _ 2 , C\}$, by the assumed optimality of $Q$ and   
 Proposition \ref{p:clarke} there exists $\lambda _i \geq 0$, with 
$\lambda _ 1 + \lambda _ 2 + \lambda _ 3 = 1$,  such that  
$$  \lambda _ 1 \nabla  \phi  _{P _1} (Q)  + \lambda_2 \nabla   \phi  _{P_2} (Q)  + \lambda _ 3  \nabla   \phi  _C (Q) = 0\,.  $$

As usual, we are now going to consider  two specific perturbations of $Q$.

First, we  let $Q _ \e$ be  the quadrilaterals obtained by rotating  $S_4$ around its midpoint:
since  $\widehat C$ is unchanged, and 
neither 
$\theta _ 1 (   |Q|    + 2 |T_1|)$  nor $|Q|$ are affected at first order, we have 
\begin{equation} \label{f:passo} 
\lambda _ 2 \frac{d}{ d \e}  \big [ (\theta _ 2 + \e)  \big (   |Q _\e |     + 2 |T _{2, \e} | \big ) \big ] \Big | _{ \e = 0} = 0\,.  
\end{equation} 
 
 We claim that $\lambda _ 2 \neq 0$. (This is easily checked by arguing by contradiction in a similar way as we did
 in the case $\mathcal P = \{ P _ 1, P _ 2 \}$ and $\mathcal V = \{ D\}$,  by considering this time
 the perturbations of $Q$ obtained by rotating $S_2$ around its midpoint so that $\widehat C$ is changed into $\widehat C + \e$). 
Thus, in  \eqref{f:passo}  we can divide by $\lambda_2$ and,   
    by arguing as in the cases previously  analysed we arrive at the equality  \eqref{f:2e}.
    
Second, if  $Q _ \e$   is   the quadrilateral obtained by rotating  $S_2$ around its midpoint, 
  so that $\widehat C$ becomes $\widehat C - \e$, we obtain 
  $$\lambda _2 \frac{d}{ d \e}  \big [ (\theta _ 2 + \e)  \big (  |Q_\e |   + 2 |T _{2, \e} | \big ) \big ] \Big | _{ \e = 0} - \lambda _ 3 =   0\,,  $$ 
 so that
  $$\frac{d}{ d \e}  \big [ (\theta _ 2 + \e)  \big (   |Q_\e |    + 2 |T _{2, \e} | \big ) \big ] \Big | _{ \e = 0}  = \frac{ \lambda _ 3}{\lambda_2}  \geq   0\,,  $$ 
  
Recalling  the asymptotic expansion (cf. \eqref{f:deriv}) 
$$T _ { 2, \e} = |T_2| -  
\frac{1}{2}   | P _ 2  B |    | P _ 2  C |  \,   \e + o ( \e)\,,$$  
 and taking into account that $ |Q_\e |  = 1 + o (\e)$, 
the above inequality leads to 
\begin{equation}\label{f:prodotto2ter}
|P _ 2 B| | P _ 2 C | \leq  \frac{1}{\theta _ 2 } \big ( 1 + 2 |T _2 | \big )  = |P _ 2 A| | P _ 2 D |  
\,. 
\end{equation} 
Since $P_2$ lies on the two secants $P_2AB$ and $P_2CD$, the inequality
\eqref{f:prodotto2ter}
is equivalent to saying that $C$ belongs to the closed disk bounded by the circle through $A,B,D$. We distinguish two cases. 
If $C$ lies on this circle, then $Q$ is cyclic, and the sum of the two opposite angles is equal to $\pi$. 
If $C$ lies strictly inside the circle, we extend the half-line $P_2C$ until it meets the circle at a point $C'$. 
Then the quadrilateral $A,B,C',D$ is cyclic, while moving the vertex from $C'$ to the interior point $C$ strictly decreases the corresponding angle. 
In both cases we infer that  \eqref{f:3e} is fulfilled.

\smallskip
\underbar{Step 4.} We have  the following identity 
   \begin{equation}\label{f:k1}
 \frac{ \widehat C^2}{\widehat C^ 2 - \theta _ 2 ^ 2} \frac{\sin ^ 2 \theta _ 2}{\theta _ 2 ^ 2}  = 
 \frac{ \sin \widehat B \, \sin \widehat C}{ \sin \widehat A \, \sin \widehat D }\,.
\end{equation} 

Indeed, since $P_2$ and $C$ are active, by Lemma \ref{l:formulas} we have
$\widehat C = \theta _2 ( 1 + 2 |T_2|)$, and hence, using \eqref{f:2e} and the equalities
$$|T_ 2| = \frac{|P _ 2 A| ^ 2}{2} \frac{\sin \theta _ 2 \sin \widehat A}{\sin \widehat B}\, , \qquad 
1 + |T_2| =  \frac{|P _ 2 D| ^ 2}{2} \frac{\sin \theta _ 2 \sin \widehat D}{\sin \widehat C}\,, $$ 
we get 
$$
\Big ( \frac{ \widehat C  }{\theta _2}  \Big )^ 2  = ( 1 + 2 |T_2|) ^ 2 = \theta _ 2 ^ 2  |P _ 2 A| ^2 | P _ 2 D | ^2  =
 \theta _ 2 ^ 2 \ \frac{2 |T_2| \, \sin \widehat B}{\sin \theta_ 2 \, \sin \widehat A} \ 
\frac{2 ( 1 + |T_2| )\, \sin \widehat C}{\sin \theta _ 2 \, \sin \widehat D} \,.
$$ 
Hence, 
 $$
 \frac{  \widehat  C  ^2}{\theta _2 ^4}  \frac{\sin ^ 2 \theta _2}{4 |T_2| ( 1 + |T_2| )  } =   \frac{ \sin \widehat B \, \sin \widehat C}{ \sin \widehat A \, \sin \widehat D }  \,,
$$ 
which gives \eqref{f:k1}  since
$\widehat C^ 2 - \theta _ 2 ^ 2 = 4 \theta _ 2 ^ 2 |T_2| ( 1 + |T_2| )$.

\medskip
 \underbar{Step 5.}  We make the following key observation:   the optimal trapezium $Q^* = A^*B^*C^*D^*$ 
  does not satisfy the equality \eqref{f:k1}. Indeed, since we have 
$$\widehat { D ^ * } = \widehat { C ^*}=\alpha^* \qquad \text{ and }  \qquad
\widehat { A ^ * } =  \widehat { B ^ * } = \pi -  \alpha ^* \,, $$
the left hand side of  \eqref{f:k1}   equals
 $$ 
 \frac{ ( \alpha^*)^2}{(  \alpha^*)^2 - (\pi- 2  \alpha^*)^2}\frac{\sin ^ 2 (\pi- 2 \alpha^*) }{(\pi- 2 \alpha^*)^ 2}\approx 2.68$$
while the right hand side  of  \eqref{f:k1} equals $1$.

\medskip
\underbar{Step 6.}  Reaching a contradiction.

 \smallskip
Recalling \eqref{f:plage} and \eqref{f:3e}, 
 in the current case of study the  angles $\widehat B, \widehat C, \widehat D$ satisfy the following relations:
 \begin{equation}\label{f:relations} 
 \widehat C \in [\alpha^* , 1.28 ] \, ,  \qquad \widehat B  > \widehat C \, , \qquad  \widehat D > \widehat C  \, ,\qquad  \widehat B + \widehat D \leq \pi \,.
 \end{equation} 
 
Moreover, due to  the current labelling of $Q$ as in Definition \ref{d:labels2},   the longest side  is either $[C,D]$ or $[C,B]$, possibly being equal. Assume that 
$|CD| \geq  |CB|$. \EEE

Let 
$\mathcal U$ denote the neighbourhood of $(\pi -  \alpha ^*, \alpha ^*, \alpha ^* )$ 
given by the family of triplets $(\widehat B,\widehat C, \widehat D)$ with
\begin{equation}\label{d10}
\widehat B \in \big  [ \pi -   { \alpha ^*} - 0.13 , \pi -   { \alpha ^*}   \big ] \, , \quad  \widehat C \in 
 \big  [    { \alpha ^*},    { \alpha ^*} + 0.1 \big ] \, , \quad  \widehat D \in  \big  [    { \alpha ^*},  { \alpha ^*} + 0.1 \big ]  \,.
 \end{equation}
 
 We are going to reach a contradiction by showing that, for  $(\widehat B,\widehat C, \widehat D)$  satisfying \eqref{f:relations}:
\begin{itemize}
\item[(i)] if  $(\widehat B,\widehat C, \widehat D) \in \mathcal U$, the identity  \eqref{f:k1} does not hold;
 \smallskip

\item[(ii)]  if $(\widehat B,\widehat C, \widehat D)  \not \in \mathcal U$, $Q$ cannot be optimal 
by virtue of Proposition  \ref{d13}. 
\end{itemize}    
  
\smallskip
 \noindent (i)   
Assume that $(\widehat B,\widehat C, \widehat D) \in \mathcal U$. Since the mapping $x \to \frac{\sin x}{x}$ is decreasing on $(0, \pi/2)$, we have \EEE 
   $$\frac{ \widehat C^2}{\widehat C^ 2 - \theta _ 2 ^ 2} \frac{\sin ^ 2 \theta _ 2}{\theta _ 2 ^ 2} \ge 
   \frac{ (  \alpha^*)^2}{ (3  \alpha^*+0.3-\pi)(\pi- \alpha^*)} \frac{\sin ^ 2 (\pi-2  \alpha^*) }{(\pi-2  \alpha^*)^ 2} >1.201$$
 and
 $$\frac{ \sin \widehat B \, \sin \widehat C}{ \sin \widehat A \, \sin \widehat D } \le \frac{   \sin ( \alpha^*+0.13)\sin ( \alpha^*+0.1)}{ \sin ( \alpha^*-0.13) \, \sin   \alpha^*}<1.195,$$
 so that the identity  \eqref{f:k1} does not hold.
 
 \smallskip
 \noindent (ii)    Assume that $(\widehat B,\widehat C, \widehat D)  \not \in \mathcal U$.

 From scale invariance, it is not restrictive to show that the rescaling of $Q$ obtained by fixing 
 $D=(0,0)$ and $C=(1,0)$ is not optimal.    
  If any of $A$ or $B$ do not belong to
   $\mathcal R _ {A }$ or $\mathcal {\mathcal R}_{B}$, respectively, then $Q$ is not optimal from Proposition   
 \ref{d13}. The remaining possibility is that $A\in \mathcal R _ {A }$ and  $B\in \mathcal {\mathcal R}_{B}$. We have
 $$  
\begin{aligned}  
&  \widehat C \le \arctan\Big (\frac{0.721923828}{1- 0.649902344}\Big)<1.12<\alpha^*+0.1\,, 
\\
\noalign{\medskip} 
&  \widehat D \le \arctan\Big (\frac{0.721923828}{0.349609375}\Big)<1.12<\alpha^*+0.1
\\
\noalign{\medskip} 
 &  \widehat B \ge \arctan\Big (\frac{1- 0.649902344}{0.721923828}\Big)+\arctan\Big (\frac{0.649902344- 0.366210938}{0.721923828-0.692138672}\Big)
 \\
 \noalign{\medskip} 
&  \quad = \pi - \arctan\Big (\frac{0.721923828}{1- 0.649902344}\Big)-\arctan\Big (\frac{0.721923828-0.692138672}{0.649902344- 0.366210938}\Big)
\\
\noalign{\medskip} 
& \quad >\pi-1.224>\pi-\alpha^*-0.13.
\end{aligned} 
$$
 This implies that $(\widehat B,\widehat C, \widehat D)  \in \mathcal U$, against our assumption.

We conclude that
$A\notin{\mathcal R}_{A}$ or $B\notin{\mathcal R}_{B}$.
Hence, $Q$ is not optimal by Proposition \ref{d13}, which applies because,
under the current assumption that
$\mathcal P=\{P_1,P_2\}$, the two candidate values in
\eqref{f:minimum} associated with the pairs of opposite sides
are both equal to $m^2(Q)$.

\bigskip 

\underbar{Cases (${\mathcal P}2{\mathcal V}2$)}:  We have to examine separately the subcases $\mathcal V = \{B,  D \}$ and $\mathcal V = \{ C, D \}$.

\medskip
$\bullet$  {\it $\mathcal P   =  \{P_1, P _2\}$  and $\mathcal V = \{ B, D \}$. }  

\smallskip
\underbar{Step 1.} It holds that
\begin{eqnarray}
& \displaystyle |DC| ^ 2 = 
\frac{(\pi - \widehat C) \,  \sin ( \widehat C +  \widehat D)     }{  (\pi - \widehat C  - \widehat D) \, \sin \widehat D \, \sin \widehat C } = |BC| ^ 2 
& \label{f:DCbis}  
\\ 
& \displaystyle |AD| ^ 2 =  \frac{\sin \big ( \frac{\widehat C}{2} \big )  } {\sin \big ( \frac{\widehat A}{2} \big )  \, \sin (\widehat D) }
& \label{f:AD}   \,.
\end{eqnarray}
To prove the first equality in \eqref{f:DCbis}, we argue exactly as done to prove \eqref{f:DC}, just replacing $C$ by $D$. 
Also the proof of   the second equality in \eqref{f:DCbis} is analogous,  working with the triangle   $\triangle(P_1, B, C)$  in place of
the triangle $\triangle(P_2, C, D)$, exploiting the fact that $B$ and $P_1$ are simultaneously active, and  finally recalling that 
by Remark \ref{r:V1} we have $\widehat B = \widehat D$. 

Using again that $\widehat B = \widehat D$, and that $ |DC|= |BC|$ (from the equality \eqref{f:DCbis} just proved),
we obtain that the quadrilateral $Q$ is symmetric about the 
diagonal $[A, C]$. 
Then the equality \eqref{f:AD} follows by imposing that  the area of 
the triangle $\triangle (A, C, D)$, expressed in terms of $|AD|$ 
and of its inner angles, is equal to $\frac{1}{2}$. 

\medskip
\underbar{Step 2}. We have
$$m ^ 2 ( Q) =   ( \pi - \widehat B - \widehat C) \frac{\sin \widehat B}  {\sin (\widehat B + \widehat C) } \,.$$ 
Indeed, we have
$$\begin{aligned} 
m ^ 2 ( Q) & = \theta _1 (   |Q|  + 2 |T_1| )  =  \theta _1 \Big (  1 + |AD| ^ 2 \frac{\sin \widehat A \, \sin \widehat D }{\sin \theta _1} \Big   ) 
 \\
 & =  \theta _1 \Big (  1 +   \frac{\sin \big ( \frac{\widehat C}{2} \big )  } {\sin \big ( \frac{\widehat A}{2} \big )  \, \sin (\widehat D) }
 \frac{\sin \widehat A \, \sin \widehat D }{\sin \theta _1} \Big   )    
  =  \theta _1 \Big (  1 + 
 \frac{  2 \sin \big ( \frac{\widehat C}{2} \big )   \cos \big ( \frac{\widehat A}{2} \big )      }{\sin \theta _1} \Big   )   
 \\ & =  \theta _1 \Big (  1 + 
 \frac{   \sin \big ( \frac{\widehat C}{2} + \frac{\widehat A}{2} \big )  +   \sin \big ( \frac{\widehat C}{2} - \frac{\widehat A}{2} \big )       }{\sin \theta _1} \Big   )  
  =  \theta _1 \Big (  1 + 
 \frac{   \sin  ( \widehat B )  -    \sin  ( \widehat B + \widehat C  )       }{\sin \theta _1} \Big   )   
 \\
 & = \theta _1 \Big (  1 + 
 \frac{   \sin  ( \widehat B )  -    \sin  ( \theta _1  )       }{\sin \theta _1} \Big   )    =  
 ( \pi - \widehat B - \widehat C) \frac{\sin \widehat B}  {\sin (\widehat B + \widehat C) } \,.
    \end{aligned}
$$
where we have used, in the order, Lemma \ref{l:formulas}, the equality \eqref{f:AD}, elementary trigonometric formulas,
 the equalities $\frac{\widehat C}{2} + \frac{\widehat A}{2}  = \pi - \widehat B$ and $\frac{\widehat C}{2} - \frac{\widehat A}{2}   = \widehat B + \widehat C - \pi$ 
 (obtained by computing the sum of the inner angles  of $Q$), and finally the equality $\widehat B + \widehat C = \pi - \theta _1$
 (obtained by computing the sum of the inner angles  of $\triangle (P _1, B, C)$). 

\medskip 
   \underbar{Step 3}. We have
   $m ^ 2 ( Q) < m ^ 2 ( Q^*)$, against the optimality of $Q$. 
   
   \smallskip
   From Step 2, since the map $ t \mapsto \frac{t}{\sin t}$ is increasing on $(0, \pi)$ and we have 
$0< 2 \widehat B < \widehat B + \widehat C  < \pi $,  it holds that
\begin{equation}\label{f:distintivo}
m ^ 2 ( Q)  \leq   \frac{ ( \pi - 2 \widehat B )  \sin \widehat B}  {\sin (2 \widehat B) }  = \frac{\pi - 2 \widehat B}  { 2 \cos  \widehat B }  = : f ( \widehat B)\,.
\end{equation}  

Recalling \eqref{f:plage}, and since $B \in \mathcal V$,  we have that $\widehat B$ belongs to the interval $I= [1.10 , 1.28 ]$. 

We observe that the map $t \mapsto f ( t)$ is decreasing on such interval. Indeed, we have
$$f' ( t )  = -\frac{1}{\cos t}+\frac{(\pi-2t)\sin t}{2\cos^2 t} < 0 \ \Longleftrightarrow \ g(t):=  \pi-2t   -2 \cot t <0 $$ 
and the inequality $g (t) <0$ holds true on  $I$ since
$$ g(1.28)  < 0 \qquad \text{ and } \qquad g' (t) = 2 \cot ^ 2 t >0 \,.$$ 
We conclude that
$$f (\widehat B ) \leq f (1.10) \leq  1.038 <   m^2  \EEE (Q^*)\,.  $$ 

\bigskip 
 $\bullet$  {\it $\mathcal P   =  \{P_1, P _2\}$  and $\mathcal V = \{ C, D \}$. }  
  
\smallskip We proceed in a similar way as in the previously analysed case when $\mathcal P   =  \{P_1, P _2\}$  and $\mathcal V = \{ C \}$. 
By repeating the arguments of Steps 1 and 2 in that case, and taking into account that, in the current case of study, also $\widehat D$ is active,  we obtain 
$$m^2 ( Q) = f ( \widehat B, \widehat C, \widehat C)\,,$$ 
where  $f$  is the function defined in \eqref{f:mdel}.

Next, we observe that the following relations analogue to 
 \eqref{f:2e}-\eqref{f:3e} hold:
\begin{eqnarray} 
 &  |P _ 1 A| | P _ 1 B | = \displaystyle  \frac{1}{\theta _ 1 } \big ( 1 + 2 |T _1 | \big )  & \label{f:1e} 
\\  \noalign{\smallskip} 
& 
\widehat B +  \widehat C \leq \pi\,. & \label{f:3ebis}  
\end{eqnarray}

Indeed, the inequality \eqref{f:3ebis} is immediately checked by considering the sum of the interior angles of $\triangle ( P_1, B, C)$,  
since $\theta _ 1 >0$. To show the equality \eqref{f:1e} we apply Proposition \ref{p:clarke}:
 since $Act ( Q) = \{ P _1, P _ 2 , C, D\}$, by the assumed optimality of $Q$,
  there exists $\lambda _i \geq 0$, with 
$\sum _{ i = 1} ^ 4 \lambda _ i = 1$,  such that  
$$\lambda _ 1 \nabla  \phi  _{P _1} (Q)  + \lambda_2 \nabla   \phi  _{P_2} (Q)  + \lambda _ 3  \nabla   \phi  _C (Q)+ \lambda _ 4  \nabla   \phi   _D (Q)  = 0\,.$$

Consider then the quadrilaterals $Q _ \e$ obtained by rotating  $S_1$ around its midpoint:
since  $\widehat C$, $\widehat D$  are unchanged, and 
neither 
$\theta _ 2 (    |Q|    + 2 |T_2|)$  nor $|Q|$ are affected at first order, we have 
\begin{equation} \label{f:passobis} 
\lambda _ 1 \frac{d}{ d \e}  \big [ (\theta _ 1 + \e)  \big (   |Q _ \e |  + 2 |T _{1, \e} | \big ) \big ] \Big | _{ \e = 0} = 0\,.  
\end{equation} 
 
Then \eqref{f:1e} follows from \eqref{f:passobis} by the same computations done in the previously analysed cases, provided we can show that 
$\lambda _ 1 \neq 0$. This  follows from a contradiction argument: assuming  that $\lambda _ 1 = 0$, and considering
 the perturbations of $Q$ obtained by rotating $S_3$ around its midpoint, we obtain
 $\lambda _ 3 = \lambda _ 4$. This forces  the equality
 $\lambda _ 2 \nabla   \phi    _{P_2}(Q) + \lambda _ 3 ( \nabla   \phi   _ C ( Q)     +    \nabla   \phi  _ D ( Q)) = 0 $ to hold under 
  an arbitrary perturbation of  $Q$. In particular, by considering a parallel inward movement of $S_3$, 
  since it does not affect $  \phi _C,  \phi  _D$, while it changes $  \phi  _{P_2}$ at first order, we obtain $\lambda _ 2 = 0$. But then 
  the simultaneous vanishing of $\lambda _1$ and $\lambda _2$  forces the equality 
  $\lambda _ 3 ( \nabla  \phi   _ C ( Q)    +   \nabla   \phi   _ D ( Q)) = 0 $ to hold under 
  an arbitrary perturbation of  $Q$.   Since $\lambda_1=\lambda_2=0$ and $\lambda_3=\lambda_4$, we have
$\lambda_3=\lambda_4=\frac12$, and hence
$\nabla\phi_C(Q)+\nabla\phi_D(Q)=0$. 
This is impossible: indeed, rotating $S_2$ around its midpoint so that
$\widehat C$ is changed into $\widehat C+\e$ leaves $\widehat D$
unchanged, and therefore
$$
\frac{d}{d\e}
\big(\phi_C(Q_\e)+\phi_D(Q_\e)\big)\Big|_{\e=0}=1\neq0.
$$

Finally,  by exploiting \eqref{f:1e}, we repeat also the argument of Step 4,  and we arrive, in place of \eqref{f:k1},  at
 \begin{equation}\label{f:kk1}
 \frac{ \widehat C^2}{\widehat C^ 2 - \theta _ 1 ^ 2} \frac{\sin ^ 2 \theta _ 1}{\theta _ 1 ^ 2}  = 
 \frac{ \sin ^2 \widehat C}{ \sin \widehat A \, \sin \widehat B }\,.
\end{equation} 
We observe that, this time, the above relation is {\it true} for the optimal trapezium $Q^*$,  in the sense that it 
is asymptotically satisfied  in the limit as $\theta_ 1 \to 0$, replacing $\widehat C$ by $\widehat { C ^*}$ and $\widehat A,  \widehat B$ by $\pi -  \widehat {C ^*}$. 
Nevertheless we observe that, since
$$\widehat B = \pi - \widehat C-  \theta _ 1\,  \qquad \text{ and } \qquad  \widehat A = 2 \pi - \widehat B - 2 \widehat C = \pi - \widehat C + \theta _ 1 \,,  $$ 
we have 
$$\sin ( \widehat B ) = \sin ( \widehat C +  \theta _ 1) \qquad   \text{ and } \qquad  \sin ( \widehat A   ) = \sin ( \widehat C -  \theta _ 1  ) \,.$$ 
Inserting these identities into \eqref{f:kk1}, such equality becomes
 \begin{equation}\label{f:kk1bis} 
 \theta _ 1 ^ 2 - \cos ^ 2 (\theta _1)  \sin ^ 2 (\theta _1) - \frac{\theta _ 1 ^ 4}{\widehat C ^ 2}  + \cot ^ 2(\widehat C )  \sin ^ 4 (\theta _1) = 0 \,. 
 \end{equation}  
which is in turn equivalent to 
 \begin{equation}\label{f:kk1bis2} \theta _ 1 ^ 2\Big (1-\frac{\sin^2 \theta_1}{\theta_1^2}\Big ) + \frac{\theta_1^4}{\sin^2 \widehat C} \Big (\frac{\sin^4 \theta_1}{\theta_1^4} - \frac{\sin^2 \widehat C}{\widehat C^2}\Big)=0.
  \end{equation}

Let 
$\mathcal U$ denote the neighbourhood of $(\pi -  \alpha ^*, \alpha ^*)$ 
given by the family of pairs $(\widehat B,\widehat C)$ with
 \begin{equation}\label{d12}
 \widehat B \in \big  [ \pi -   \alpha^* - 0.13 , \pi -   \alpha^*   \big ] \, , \quad  \widehat C \in 
 \big  [   \alpha^*,   \alpha^* + 0.1 \big ]  \,.
 \end{equation}
 
 We are going to reach a contradiction by showing that:
 \begin{itemize}
\item[(i)] if  $(\widehat B,\widehat C) \in \mathcal U$, the identity   \eqref{f:kk1bis2} does not hold;
 \smallskip

\item[(ii)]  if $(\widehat B,\widehat C)  \not \in \mathcal U$, $Q$ cannot be optimal 
by virtue of Proposition \ref{d13}. 
\end{itemize}

 \noindent (i) Assume that $(\widehat B,\widehat C) \in \mathcal U$.  
Then the identity \eqref{f:kk1bis2} cannot hold because, 
 on the one hand, 
 $$\forall \theta _1 \in (0, \frac \pi2),\; \; \; \; 1-\frac{\sin^2 \theta_1}{\theta_1^2} \ge 0,$$
 while, on the other hand, 
 $$\forall \theta_1\in (0, 0.13],\;\; \frac{\sin^4 \theta_1}{\theta_1^4} \ge \frac{\sin^4 0.13}{0.13^4}>0.98$$
 and
 $$\forall  \widehat C \in [\alpha^*, \alpha^*+0.1), \;\;\;\; \frac{\sin^2 \widehat C}{\widehat C^2} <\frac{\sin^2 \alpha^*}{(\alpha^* )^2}<0.66\,. $$

  \noindent (ii)  Assume that $(\widehat B,\widehat C) 
 \not\in \mathcal U$.  If $ |CD| \ge |CB|$, then the fact that $Q$ cannot be optimal follows by repeating verbatim the argument used to prove the assertion (ii) in Step 6 of  the case  $\mathcal P   =  \{P_1, P _2\}$ and  $ \mathcal V  = \{ C \}$.     If $|CD| < 
 |CB|$, we switch the labels $D$ and $B$, we align the new side $[C,D]$ 
(formerly $[B, C]$)  on the horizontal axis, and we rescale $Q$ so that such side has unit length. 
Then the fact that this rescaling of $Q$ is not optimal  follows from Proposition \ref{d13}, since there is no quadrilateral, with $\widehat B= \widehat C$, such that 
 $A \in \mathcal R_{A^*}$   and $B \in {\mathcal R}_{B^*}$. 
(Notice that Proposition \ref{d13} can be applied since
$\mathcal P=\{P_1,P_2\}$; indeed, before the relabelling the two
candidate values in \eqref{f:minimum} associated with the pairs
of opposite sides are both equal to $m^2(Q)$, while
the relabelling only interchanges the two pairs of opposite sides,
and the subsequent rescaling preserves their equality.)  
 
 \medskip

 \bigskip 
\underbar{Case (${\mathcal P}2{\mathcal V}3$)}:  {\it $\mathcal P   =  \{P_1, P _2\}$    and  $\mathcal V = \{ B, C, D \}$.  }   

\smallskip

  We can repeat the same proof adopted when 
  $ {card}  (\mathcal P )  = 2$ and  $ {card}  (\mathcal V ) = 2$,  in the case
 $\mathcal V = \{ B, D \}$. The unique difference, which does not affect the conclusion,  
 is that the inequality \eqref{f:distintivo} holds with equality sign.

\bigskip\bigskip

\section{Proof of Theorem  \ref{t:Qstar} }\label{sec:trap}
 
 Assume that $Q$ is a solution to problem \eqref{f:pb}, and assume it has exactly two parallel sides. 
    Below we examine separately each of the configurations 
in the table given in Section \ref{sec:setup},   
    and we show that all of them lead to  a contradiction, except for the last one, which leads to identify $Q$ with $Q^*$. 
    As usual, we work under the not restrictive normalization   $|Q | = 1$.

\medskip 
\underbar{Case $({\mathcal P}1{\mathcal V}2$)}:   {\it $\mathcal P = \{P _ 1 ^ \infty\}$  and   $\mathcal V =\big \{C, D \big \}$}.

\smallskip 
We can argue as done    to handle the case   $({\mathcal P}1{\mathcal V}2)$ 
 in Section \ref{sec:notrap}. \EEE 

\medskip
\underbar{Case $({\mathcal P}2{\mathcal V}0$)}:     {\it $\mathcal P = \{P _ 1 ^ \infty, P_2\}$ and $\mathcal V = \emptyset$}.

\smallskip

Denoting by $ d  $ the distance between the two parallel sides $S_1$ and $S_3$, 
we have 
\begin{equation}\label{f:dueemme} 
m^2(Q)=d^2 = \theta_2 (1+2|T_2|)
=\theta _2 \left(1+\frac{\ell_1^2}{S}\right)\,.
\end{equation} 
The triangle $T_2$ has base $S_1$, base angles $\widehat D$ and $\widehat C$, and opening angle
$\theta _2=\pi-\widehat C- \widehat D$. Thus, we have
$$
|T_2|=\frac{\ell_1^2}{2}\frac{\sin \widehat C \sin \widehat D}{\sin(\widehat C +\widehat D)} = \frac{\ell_1^2}{2S} \,,
$$
where
$$
S:=\cot \widehat C+\cot \widehat D
$$
(in particular, $S>0$). Since
$$
\ell_3-\ell_1=d(\cot \widehat C+\cot \widehat D)=dS\,,
$$
and, from the equality  $|Q|=1$, 
$$
1=\frac{\ell_1+\ell_3}{2}\,d
=\ell_1d+\frac{d^2S}{2}\,,
$$
the length $\ell _1$ can be computed in terms of $d$ and $S$ as 
$$
\ell_1=\frac1d-\frac{dS}{2}.
$$

Substituting the above expression of $\ell_1$ into \eqref{f:dueemme}, we find
$$
d^2=\theta _2 \left(1+\frac1S\left(\frac1d-\frac{dS}{2}\right)^2\right).
$$
Therefore
$$
d^2=\theta _2 \left(\frac1{Sd^2}+\frac{Sd^2}{4}\right),
$$
and consequently
$$
d^4=\frac{\theta _2}{S\left(1-\frac{\theta _2 S}{4}\right)}.
$$
We observe that, for fixed $\theta_2 >0$, the function
$$
S\mapsto \frac{\theta_2}{S\left(1-\frac{\theta_2 S}{4}\right)}
$$
is strictly decreasing whenever $\theta_2 S<2$.  Let us check that the latter inequality is satisfied. 
 Denoting by $\alpha^*$ the acute angle of the optimal trapezium $Q^*$,   we have
\begin{equation}\label{f:minalfa} 
\min\{\widehat C,\widehat D\} > m^2(Q) \ge m^2(Q^*)=\alpha^*\,, 
\end{equation}
where the first inequality holds since $\mathcal V=\emptyset$,  and the second one by the assumed optimality of $Q$. 
Consequently we have
$$
\theta _2 =\pi- \widehat C - \widehat D \leq \pi - 2 \alpha ^* .
$$
Moreover, since the cotangent is decreasing on $(0,\pi)$, we get
$$
S=\cot \widehat C+\cot \widehat D \leq 2\cot \alpha^* .
$$
It follows that
$$
\theta_2 S
\leq 2(\pi-2\alpha^*)\cot\alpha^*<2.
$$
The last inequality follows from the defining equation of $\alpha^*$, or equivalently by a direct computation at the value of the acute angle of $Q^*$.

We claim that
\begin{equation}\label{f:claimS}
S\geq 2\cot y,
\qquad \text{ with } 
y:=\frac{\widehat C+\widehat D}{2}.
\end{equation}
Indeed, since $\widehat C,\widehat D\in(0,\pi)$ and
$\widehat C+\widehat D=2y<\pi$, we have
$$
S= \cot \widehat C+\cot \widehat D =
\frac{\sin(\widehat C+\widehat D)}
{\sin\widehat C\,\sin\widehat D}
=\frac{\sin(2y)}
{\sin\widehat C\,\sin\widehat D}\,.$$
Moreover,
$$
\sin\widehat C\,\sin\widehat D
= \frac{\cos(\widehat C-\widehat D)-\cos(\widehat C+\widehat D)}{2}
\leq
\frac{1-\cos(2y)}{2} =
\sin^2 y\,.$$ 
Using \eqref{f:claimS} and 
the equality 
$$
\theta_2=\pi-\widehat C-\widehat D=\pi-2y,
$$
we infer that
$$
d^4
=
\frac{\theta _2}{S\left(1-\frac{\theta _2 S}{4}\right)}
\leq
 \frac{\pi-2y}{2\cot y\left(1-\frac{(\pi-2y)\cot y}{2}\right)} =: \Psi (y)\,.
$$

From \eqref{f:minalfa}, we have $y \in (\alpha^*,\pi/2)$.  We observe that the function
$
y\mapsto \Psi (y) $
is strictly decreasing on $[\alpha^*,\pi/2)$. Indeed, setting
$r:=\pi-2y$,
it can be rewritten as
$$
\frac{r}{2\tan(r/2)-r\tan^2(r/2)}\,, 
$$
which  is increasing with respect to $r$ on $(0, \pi - 2 \alpha ^ *)$, 
 as follows by direct differentiation, 
since $(0, \pi - 2 \alpha ^ *) \subset (0, \frac{\pi}{2})\,. $ 

We deduce that
$$
d ^ 4  < \Psi (\alpha ^* ) = \frac{\pi-2\alpha^*}{2\cot\alpha^*
\left(1-\frac{(\pi-2\alpha^*)\cot\alpha^*}{2}\right)}
=
(\alpha^*)^2\,,
$$
where the last equality holds by the definition of $\alpha ^*$. 

This contradicts
$$
m^2(Q)=d^2\geq m^2(Q^*)=\alpha^*.
$$

\bigskip
\underbar{Case $({\mathcal P}2{\mathcal V}1$)}:   {\it $\mathcal P=\{P_1^\infty,P_2\}$ and $\mathcal V=\{D\}$.}

\smallskip 
We can repeat verbatim the argument used to handle the previous case, because therein 
the condition $\mathcal V=\emptyset$ was used only to write \eqref{f:minalfa}. 
In the present case, since $D\in\mathcal V$ whereas $C\notin\mathcal V$, by Lemma 
\ref{l:formulas} and Remark \ref{r:V2} we have
\begin{equation}\label{f:minalfa2} 
\widehat C>\widehat D=m^2(Q)\ge m^2(Q^*)=\alpha^* .
\end{equation}
This is sufficient to ensure that, setting 
$
y:=\frac{\widehat C+\widehat D}{2},
$
we still have $y>\alpha^*$. Moreover, since $\theta_2=\pi-\widehat C-\widehat D>0$, we still have 
$y<\pi/2$. Hence $y\in(\alpha^*,\pi/2)$, and the same argument as above gives again a contradiction.

\bigskip 
\underbar{Case $({\mathcal P}2{\mathcal V}2$)}:   {\it  $\mathcal P = \{P _ 1 ^ \infty, P_2\}$ and   $\mathcal V = \big \{C, D \big \}$}. 
 
 \smallskip
Since $C,D\in \mathcal V$,   we have that $Q$ is an isosceles trapezium, with
$$
\alpha:= \widehat C=\widehat D  <
\widehat A=\widehat B = \pi - \alpha\,.
$$

Let $\ d  $ be the distance between the two parallel sides $S_1$ and $S_3$. Since $P^\infty_1\in \mathcal P$, 
and $C, D \in \mathcal V$, by Lemma \ref{l:formulas} we have 
\begin{equation}\label{f:m1bis} 
m^2(Q)=\alpha  = d ^ 2\,.
\end{equation} 

Since $Q$ is isosceles with acute base angles equal to $\alpha$,  
 we have
$$
\ell_3-\ell_1=2d\cot\alpha .
$$
Moreover, from $|Q|=1$,
$$
1=\frac{\ell_1+\ell_3}{2}\, d .
$$
The previous three equations allow to compute $\ell _1$ in terms of $\alpha$: 
$$
\ell_1=\frac{1}{\sqrt\alpha}-\sqrt\alpha\cot\alpha
=\frac{1-\alpha\cot\alpha}{\sqrt\alpha}.
$$

We now use the fact that $P_2\in \mathcal P$.  We have 
\begin{equation}\label{f:m2bis} 
m^2(Q)=\theta_2\bigl(  |Q|   +2|T_2|\bigr) = \theta _ 2 \Big (  1 +  \frac{\ell_1^2}{2}\tan\alpha  \Big )  \,,
\end{equation} 
where the first equality holds by Lemma \ref{l:formulas} (ii), and the second one since 
$T_2$ is an isosceles triangle on the base $S_1$, with base angles equal to $\alpha$, so that its opening angle at $P _ 2$ is  
$\theta_2=\pi-2\alpha$.  

By combining \eqref{f:m1bis} and \eqref{f:m2bis}, and  substituting the above expression of $\ell_1$, we obtain 
$$
\alpha=(\pi-2\alpha)\left(1+\frac{\ell_1^2}{2}\tan\alpha\right) =(\pi-2\alpha)\left(1+\frac{\tan\alpha}{2\alpha}(1-\alpha\cot\alpha)^2\right).
 $$

Some straightforward computations show that the above equation is equivalent to 
the equation which characterizes the acute base angle of the optimal trapezium $Q^*$, namely
 \begin{equation}\label{f:alfastar} 
\left(\frac\pi2-\alpha\right)(\tan^2\alpha+\alpha^2)
=\alpha^2\tan\alpha \,.\end{equation}

Provided we show that 
 the above equation has a unique solution  $  \alpha ^* $ in the admissible interval $\alpha\in(0,\pi/2)$, 
we have that, up to rigid motions, $Q= Q^*$.  Indeed, we have  
$$  
h=\sqrt\alpha ^*,\qquad
\ell_1=\frac{1-\alpha ^* \cot\alpha ^* }{\sqrt{\alpha ^*}},\qquad
\ell_3=\frac{1+\alpha^* \cot{\alpha ^* }}{\sqrt{\alpha ^*}}.  
$$
 
 Therefore,  our proof is achieved thanks to the following
 
 \smallskip
 {\it Claim}: Equation \eqref{f:alfastar} admits a unique solution in the admissible interval $\alpha\in(0,\pi/2)$.  
  
 Namely,  
 set
$$
F(\alpha):=\left(\frac{\pi}{2}-\alpha\right)(\tan^2\alpha+\alpha^2)-\alpha^2\tan\alpha .
$$
  Any zero of $F$ belongs necessarily to $(\pi/3,\pi/2)$. Indeed, if $F(\alpha)=0$, then
$$
\alpha^2\tan\alpha=\left(\frac{\pi}{2}-\alpha\right)(\tan^2\alpha+\alpha^2).
$$
Since $\tan\alpha>\alpha$ for $\alpha\in(0,\pi/2)$, we have
$\tan^2\alpha+\alpha^2>2\alpha\tan\alpha\,,$
and hence
$$
\alpha^2\tan\alpha>2\alpha\left(\frac{\pi}{2}-\alpha\right)\tan\alpha,
$$
so that $\alpha>\pi/3$.
Now, a direct differentiation, followed by elementary simplifications using
$\alpha\in(\pi/3,\pi/2)$, show that $F'(\alpha)<0$ throughout this interval, so that  $F$ has at most one zero in $(\pi/3,\pi/2)$.
 Since
$$
F\left(\frac{\pi}{3}\right)>0 
\quad \text{ and } \quad 
\lim_{\alpha\to(\pi/2)^-}F(\alpha)=-\infty,
$$
such a zero exists and is unique.   Thus the claim holds true, and our proof is achieved.  
 \qed

\section{Proof of Proposition \ref{d13} }\label{d11}

We use a computational framework based on interval arithmetic, which allows us to prove that quadrilaterals lying sufficiently far (quantified explicitly) from the conjectured optimal trapezium cannot be optimal. Techniques of this kind are commonly referred to in the literature as certified computing, validated numerics, or interval arithmetic. Introductory material and further references can be found in \cite{Rump2010,Tucker2011}.

In interval arithmetic, floating-point
numbers $x$ are replaced by intervals, denoted by $[x]$, and arithmetic
operations on intervals are defined so that the resulting interval contains
all possible values obtained by applying the operation to
elements of the input intervals. For instance, the product interval
$[x][y]$ contains every product $st$ with $s\in[x]$ and $t\in[y]$.

From a mathematical point of view, computations based on validated numerics
can therefore provide rigorous certificates that a given quantity lies below
a prescribed threshold, or that two numerical quantities are distinct. We
refer to the aforementioned works for further details and applications.

Several numerical packages implement interval arithmetic. We mention INTLAB
\cite{Rump2010} and the open-source FLINT library
\begin{center}
	\href{https://flintlib.org/}{\nolinkurl{https://flintlib.org/}}
\end{center}
which is used in the present work.

The space of quadrilaterals with fixed area, identified up to rigid motions, has four degrees 
of freedom. 
Since working under a fixed-area constraint is more challenging for a 
quadrilateral exploration algorithm, 
we prefer to fix the endpoints $C,D$ of the longest side 
of the quadrilateral by setting $D=(0,0)$ and $C=(1,0)$, so that 
$|CD|=1$. 
The assumption that $[C, D]$ is the longest side constrains the remaining 
vertices $A,B$ to lie in the unit half-disks centered at $D, C$, 
respectively.

 For any such quadrilateral, formulas \eqref{eq:fence-formulas} 
provide the six candidate values. 
The formulas for $\phi _A, 
\phi _B, \phi _C$, and $\phi _D$  are straightforward and do not cause 
numerical difficulties, except when the area is small, a case that 
can be  excluded a priori thanks to Lemma \ref{l:existence} and its proof. Concerning the formulas for $\phi _{P_1}$ and $\phi _{P_2}$,    we distinguish two 
cases:	

\smallskip 
\begin{itemize}
		\item[--] When a pair of opposite sides, either  $(S_1, S_3)$ or  $(S_2, S_4)$,  is  {\it  ``far from being parallel'':}  
the corresponding point  $P_i$  is computed as the intersection of their supporting lines, 
and  the corresponding value $\phi_{P_i}$ is obtained from
\eqref{eq:fence-formulas}.

		\item [--] When a pair of opposite sides,    say  $(S_1, S_3)$,      
		are {\it almost parallel},  their 
		 intersection point   $P_1$   cannot be reliably computed numerically.		  
		  Then the formula  for $\phi _{P_1}$ is recast as follows: 
		 letting $L_1$  and $L_3$ be the supporting lines of $S_1$ and $S_3$, 
		 and letting ${X_1 , Y_1} $  lie on $L_1  $ and  $X_3 ,Y_3 $    on $ L _3 $, we have: 
		 
		\[  \phi_{P_1}    =
		 \frac{\theta }{2\sin \theta }  
				\frac{d(X_1,  L_3  )d(Y_3,    L_1  )+d(Y_1,  L _3   )d(X_3,   L_1 )}{|Q|}.\]
		This formula is stable near $ \theta  = 0$, where the series expansion of $ \theta \mapsto \frac{\theta}{\sin \theta}  $ is used in the numerical computations.
	\end{itemize}
Then the six candidate values are evaluated on rectangular boxes using
interval arithmetic. Recall that, as observed above, quadrilaterals with
sufficiently small area cannot be competitors. Moreover, by the hypothesis
of Proposition \ref{d13}, the two candidate values associated with the pairs
of opposite sides must coincide.

 Then the {\bf certification process} proceeds as follows. 

First of all we pick a threshold 
$\sigma$ 
smaller than the conjectured value of the maximal shortest fence, 
namely $1.049685815487323$. 
In the numerical simulations we take 
$\sigma = 1.0496$. 
We work with squared fence formulas, thus avoiding unnecessary 
square root computations. The input threshold $\sigma$ is, on the other hand, the fixed threshold for $m$.

\begin{itemize}[noitemsep]

\item We start from the initial rectangles
${\mathcal R} _ {A}  = [-1,1]\times[0,1]$
and
${\mathcal R} _ {B}  = [0,2]\times[0,1]$.
These rectangles are guaranteed to contain the vertices $A,B$ of any
quadrilateral $Q$ with longest side  $[C, D]$, being $D = (0 , 0)$ and  $C = (1 , 0)$.

\item Given a pair of rectangles
${\mathcal R} _ {A} 
= [a_1^-,a_1^+]\times[a_2^-,a_2^+]$
and
${\mathcal R} _ {B} 
= [b_1^-,b_1^+]\times[b_2^-,b_2^+]$,
we evaluate the six candidate values over the interval box 
${\mathcal R} _ {A} 
\times
{\mathcal R} _ {B}$.
If the smallest of the resulting upper bounds is at most the prescribed
threshold, then the pair
${\mathcal R} _ {A} 
\times
{\mathcal R} _ {B} $
is {\bf discarded}: no competitor for the quadrilateral with maximal
shortest fence can lie in that region.
Otherwise, we say that the rectangle pair {\bf survives} the test.

\item For a {\bf survivor} pair, we compute the scaled widths
\[
\frac{a_1^+-a_1^-}{2},
\quad
a_2^+-a_2^-,
\quad
\frac{b_1^+-b_1^-}{2},
\quad
b_2^+-b_2^-.
\]
The coordinate with the largest scaled width is bisected.
If an $a$ coordinate is selected, only
${\mathcal R} _ {A} $
is modified; if a $b$ coordinate is selected, only
${\mathcal R} _ {B} $
is modified.

\item A box is split until one of the following occurs:

\begin{itemize}

\item it is {\bf discarded} as certainly inadmissible:
one side length is definitely greater than $1$, or convexity is
violated;

\item it is {\bf certified at most $\sigma$}:
the upper bound of one of the corresponding interval enclosures
is at most $\sigma$;

\item it is {\bf certified low/flat area}:
the quadrilaterals in the corresponding region have area too small to
be competitors, by the argument given above (estimate \eqref{f:estimatearea} in the proof of Lemma \ref{l:existence});

\item it is {\bf certified incompatible}:
if the two opposite-side candidate values give non-overlapping
intervals, the quadrilaterals in the box cannot satisfy the equality
assumption in Proposition \ref{d13}.

\item it reaches {\bf a width floor}, prescribed as a parameter:
if the maximum scaled width is smaller or equal to the
\texttt{wfloor} parameter given as input to the algorithm, then the box
is labelled a {\bf survivor}, since it could not be eliminated by any
of the previous tests.

\end{itemize}

\end{itemize} 

{\bf Experiments and results.} The threshold used in the computations is
$\sigma = 1.0496$.
The initial computations are performed with the width-floor parameter
\texttt{wfloor=0.025}, followed by six successive refinements.
The results are summarized in Table \ref{tab:results-refinements}.
The reported computation times refer to a computer equipped with
a 6-core Intel i7 processor and 32 GB of RAM. 

\begin{table} \label{tab:results-refinements}
	\scriptsize
	\begin{tabular}{lrrrrrrrrr}
		\hline
		Run & \texttt{wfloor} & Time (s) & Leaves & Discarded & Certified &
		Flat &  Incompat   & Survivors & Volume \\
		\hline
		Base & 0.025 & 79.139 & 215758 & 11555 & 116319 & 213 & 17884 & 69787 &
		\(1.66{\times}10^{-2}\) \\
		Refine 1 & 0.0125 & 117.585 & 360798 & 2967 & 282735 & 0 & 50065 & 25031 &
		\(3.72{\times}10^{-4}\) \\
		Refine 2 & 0.00625 & 35.091 & 119594 & 28 & 51386 & 0 & 55878 & 12302 &
		\(1.14{\times}10^{-5}\) \\
		Refine 3 & 0.003125 & 19.360 & 66134 & 0 & 17312 & 0 & 39816 & 9006 &
		\(5.24{\times}10^{-7}\) \\
		Refine 4 & 0.0015625 & 14.399 & 49864 & 0 & 10580 & 0 & 31046 & 8238 &
		\(2.99{\times}10^{-8}\) \\
		Refine 5 & 0.00078125 & 14.503 & 50504 & 0 & 9696 & 0 & 29790 & 11018 &
		\(2.50{\times}10^{-9}\) \\
		Refine 6 & 0.000390625 & 22.381 & 72276 & 0 & 12739 & 0 & 41012 & 18525 &
		\(2.63{\times}10^{-10}\) \\
		\hline
	\end{tabular}
	
	\smallskip
	\caption{Numerical certification results: base run followed by six refinements. The counts for the different certification flags are shown.}
\end{table}

From a quantitative point of view, the total volume of the ``survivor"
rectangle pairs is less relevant than their distance from the conjectured
optimum. For each of the vertices $A$ and $B$, we consider the smallest
rectangle containing all the corresponding survivor rectangles. The resulting
rectangles are shown in Figure \ref{fig:end-rectangles}.

Any quadrilateral with {\bf at least one of the vertices} $A$ or $B$
outside these rectangles has shortest fence no larger than
$\sigma = 1.0496$, or fails to satisfy the equality assumption
in Proposition \ref{d13}. 

\begin{figure}
	\centering 
	\includegraphics[width=0.99\textwidth]{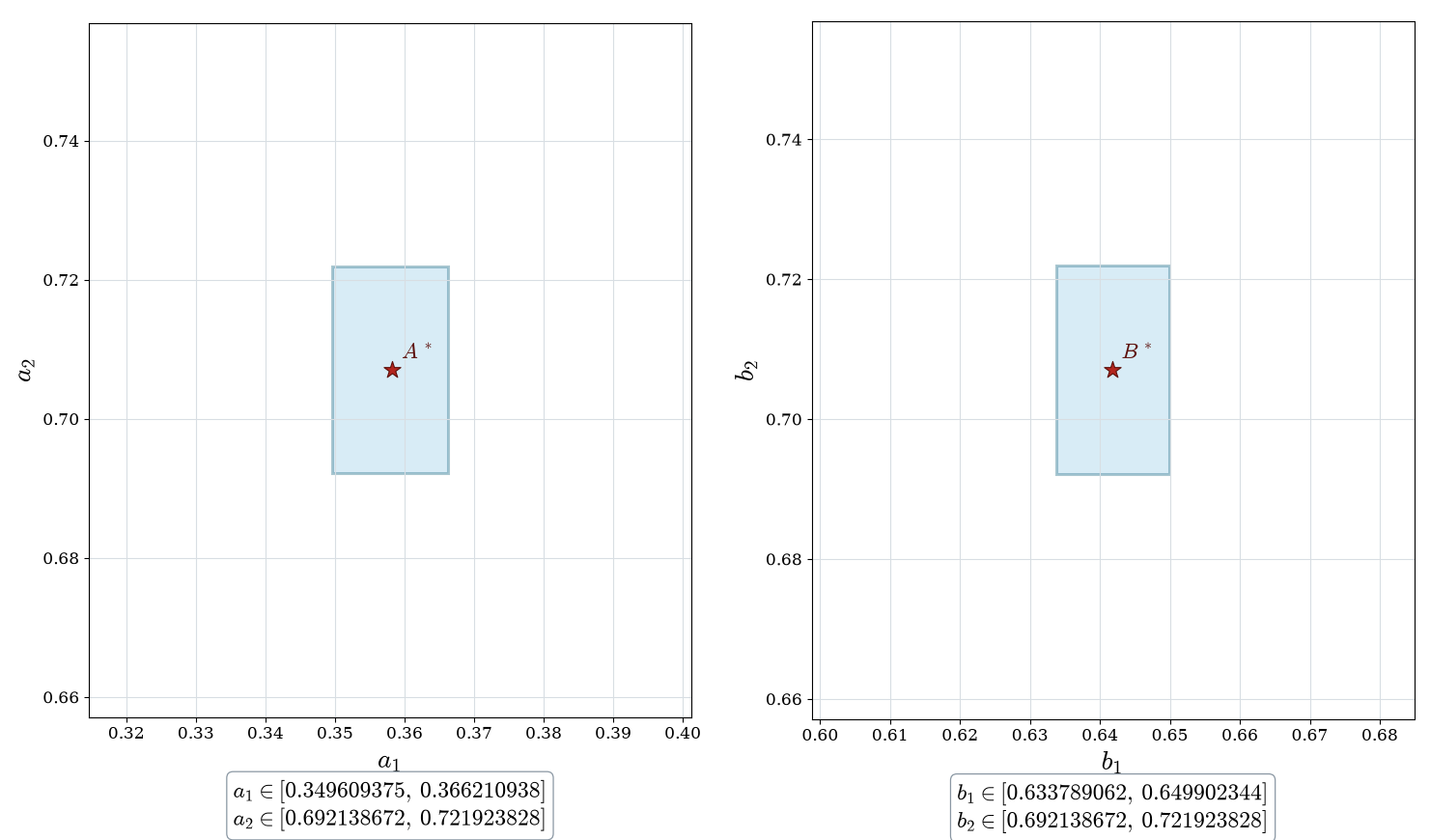}
	\caption{Rectangles containing survivor configurations for the certification procedures: if $A$ or $B$ fall outside the  illustrated rectangles, then $ABCD$ has shortest fence at most $\sigma = 1.0496$ or does not verify the equality assumption in Proposition \ref{d13}.}
	\label{fig:end-rectangles}
\end{figure}

The survivor union is contained in the product
${\mathcal R} _ {A}  \times {\mathcal R} _ {B} $, 
where
\[
\begin{aligned}
	{\mathcal R} _ {A}   &=[0.349609375,\ 0.366210938]\times[0.692138672,\ 0.721923828],\\
	{\mathcal R} _ {B}   &=[0.633789062,\ 0.649902344]\times[0.692138672,\ 0.721923828].
\end{aligned}
\]

The code allowing to do the certification procedure is available at the following Github repository:
\begin{center}
	\href{https://github.com/beniamin-bogosel/TrapeziumFence}{\nolinkurl{https://github.com/beniamin-bogosel/TrapeziumFence}}
\end{center}
It has a full documentation indicating commands used to obtain the results reported in this article. 

{\bf Additional Intlab implementation.} The FLINT/Arb implementation referenced above is efficient and certifies enclosures needed in Proposition \ref{d13}. We also implemented the certification procedure in the Matlab toolbox INTLAB \cite{Rump2010}. The equivalent code being slower to execute than the FLINT/Arb code, we introduced additional theoretical reductions further reducing the search space. These can be found in the \texttt{Theory} subfolder of the repository
\begin{center}
\href{https://github.com/beniamin-bogosel/TrapeziumFence/tree/master/Intlab}{\nolinkurl{https://github.com/beniamin-bogosel/TrapeziumFence/tree/master/Intlab}}
\end{center}

 \section*{Acknowledgments}
  We thank Andrea Cianchi for kindly informing us about a misprint in the expression given in \cite{cianchi89BUMI} for $N$ odd  of the map  $f(N)$ in \eqref{f:mapf},   and about the corrected formula contained in 
\cite{cianchitesi}.

\section*{
Conflict of interest and data availability.} The authors declare that they have no conflict of interest.
No datasets were generated or analyzed during the current study. Codes used for the certification process are made available in the linked repositories.

\end{document}